\documentclass[microtype]{gtpart}
\agtart 
\usepackage{mathrsfs,amscd}
\usepackage{pb-diagram}
\usepackage[all]{xy}
\usepackage{verbatim}

\title[Algebraic $K$-theory over the infinite dihedral group: an algebraic approach]{Algebraic $K$-theory over the infinite dihedral group:\\ an algebraic approach}

\author[J\,F Davis]{James F Davis}
\givenname{James F}
\surname{Davis}
\address{Department of Mathematics, Indiana University, Bloomington, IN 47405 U.S.A.}
\email{jfdavis@indiana.edu}

\author[Q Khan]{Qayum Khan}
\givenname{Qayum}
\surname{Khan}
\address{Department of Mathematics, University of Notre Dame, Notre Dame, IN 46556 U.S.A.}
\email{qkhan@nd.edu}

\author[A Ranicki]{Andrew Ranicki}
\givenname{Andrew}
\surname{Ranicki}
\address{School of Mathematics, University of Edinburgh, Edinburgh EH9 3JZ, SCOTLAND, U.K.}
\email{a.ranicki@ed.ac.uk}


\arxivreference{}
\arxivpassword{}

\volumenumber{11}
\issuenumber{4}
\publicationyear{2011}
\papernumber{}
\startpage{2391}
\endpage{2436}
\doi{}
\MR{}
\Zbl{}
\received{17 August 2010}
\revised{28 June 2011}
\accepted{26 July 2011}
\published{5 September 2011}
\publishedonline{}
\proposed{}
\seconded{}
\corresponding{}
\editor{}
\version{}

\newtheorem{thm}{Theorem}[section]
\newtheorem{lem}{Lemma}[section]
\newtheorem{cor}{Corollary}[section]
\newtheorem{prop}{Proposition}[section]
\makeatletter
 \let\c@lem=\c@thm
 \let\c@cor=\c@thm
 \let\c@prop=\c@thm
\makeatother

\theoremstyle{definition}
\newtheorem{conj}{Conjecture}[section]
\newtheorem{dfn}{Definition}[section]
\newtheorem{rem}{Remark}[section]
\newtheorem{exm}{Example}[section]
\newtheorem{stt}{Statement}[section]
\newtheorem*{tp}{Transitivity Principle}
\makeatletter
 \let\c@conj=\c@thm
 \let\c@dfn=\c@thm
 \let\c@rem=\c@thm
 \let\c@exm=\c@thm
 \let\c@stt=\c@thm
\makeatother

\newcommand{\di}{\displaystyle}

\newcommand{\N}{\mathbb{N}}

\newcommand{\cF}{\mathcal{F}}
\newcommand{\cG}{\mathcal{G}}
\newcommand{\cM}{\mathcal{M}}

\newcommand{\fB}{\mathscr{B}}
\newcommand{\fC}{\mathscr{C}}
\newcommand{\fE}{\mathscr{E}}
\newcommand{\fK}{\mathscr{K}}
\newcommand{\fM}{\mathscr{M}}

\newcommand{\bE}{\mathbf{E}}
\newcommand{\bK}{\mathbf{K}}

\DeclareMathAlphabet{\matheurm}{U}{eur}{m}{n}
\newcommand{\all}{\matheurm{all}}
\newcommand{\fac}{\matheurm{fac}}
\newcommand{\fbc}{\matheurm{fbc}}
\newcommand{\fin}{\matheurm{fin}}
\newcommand{\Spectra}{\matheurm{Spectra}}
\newcommand{\sub}{\matheurm{sub}}
\newcommand{\vc}{\matheurm{vc}}

\newcommand{\longra}{\longrightarrow}

\newcommand{\xo}{\otimes}
\newcommand{\x}{\times}
\newcommand{\ol}[1]{\overline{#1}}
\newcommand{\wt}[1]{\widetilde{#1}}
\newcommand{\SmMatrix}[1]{\left(\begin{smallmatrix}#1\end{smallmatrix}\right)}
\newcommand{\g}{\Gamma}
\newcommand{\eps}{\varepsilon}
\newcommand{\inv}{^{-1}}

\def\map{\mathop{\mathrm{map}}\nolimits}
\def\Nil{\mathop{\mathrm{Nil}}\nolimits}
\def\NIL{\mathop{\mathrm{NIL}}\nolimits}
\def\Or{\mathop{\mathrm{Or}}\nolimits}
\def\PROJ{\mathop{\mathrm{PROJ}}\nolimits}
\def\Wh{\mathop{\mathrm{Wh}}\nolimits}

\DeclareMathOperator*{\colim}{colim}

\begin{document}
\begin{abstract}

Two types of Nil-groups arise in the codimension 1 splitting obstruction theory for homotopy equivalences of finite CW-complexes: the Farrell--Bass Nil-groups in the  non-separating case when the fundamental group is an HNN extension  and the Waldhausen Nil-groups in the separating case when the fundamental group is 
an amalgamated free product. We obtain a general Nil-Nil theorem in algebraic $K$-theory relating the two types of Nil-groups.
 
The infinite dihedral group is  a free product and has an index 2 subgroup which is an HNN extension, 
so both cases arise if the fundamental group surjects onto the infinite dihedral group.   The Nil-Nil theorem  implies that the two types of the reduced $\wt{\Nil}$-groups arising from such a fundamental group  are isomorphic.
There is also  a topological application: in the finite-index case of an amalgamated free product, a homotopy equivalence of  finite CW-complexes is semi-split along a separating subcomplex.


\end{abstract}

\maketitle

\section*{Introduction}


The infinite dihedral group is both a free product and an extension
of the infinite cyclic group $\Z$ by the cyclic group  $\Z_2$
of order 2
\[
D_{\infty}~=~\Z_2*\Z_2~=~\Z\rtimes \Z_2
\]
with $\Z_2$ acting on $\Z$ by $-1$. A group $G$ is said to be \textbf{over $D_\infty$} if it is equipped with an epimorphism $p:\,G \to
D_{\infty}$. We study the algebraic $K$-theory of $R[G]$, for any
ring $R$ and any group $G$ over $D_{\infty}$. Such a group $G$
inherits from $D_{\infty}$ an injective amalgamated free product
structure $G=G_1*_HG_2$ with $H$ an index 2 subgroup of $G_1$ and
$G_2$. Furthermore, there is a canonical index 2 subgroup $\ol{G}
\subset G$ with an injective HNN structure
$\ol{G}=H\rtimes_{\alpha}\Z$ for an automorphism $\alpha:\,H\to
H$. The various groups fit into a commutative braid of short exact
sequences:
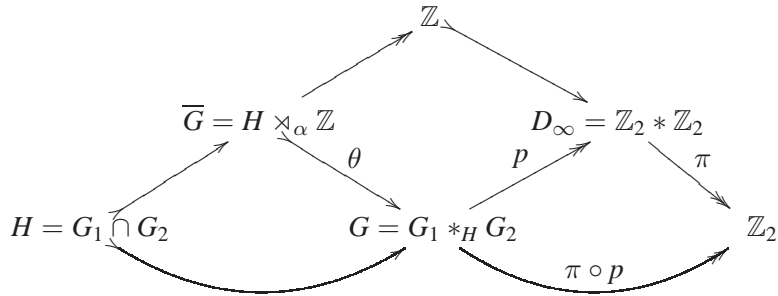
\[
\xymatrix@C-25pt{ &&\Z~\ar@{>->}[dr] && \\
& \ol{G}=H\rtimes_{\alpha}\Z\ar@{->>}[ur] \ar@{>->}[dr]^-{\di{\theta}}
&& D_{\infty}=\Z_2*\Z_2~~ \ar@{->>}[dr]^-{\di{\pi}}\\
H=G_1 \cap G_2 \ar@{>->}[ur] \ar@{>->}@/_2pc/[rr]&&
G=G_1*_HG_2\ar@{->>}[ur]^-{\di{p}}\ar@{->>}@/_2pc/[rr]^-{\di{\pi \circ p}} &&
~~\Z_2 }
\]

The algebraic $K$-theory decomposition theorems of Waldhausen for
injective amalgamated free products and HNN extensions give
\begin{equation}\label{eqn:Kamalg}
K_*(R[G])~=~K_*(R[H] \to R[G_1] \times R[G_2]) \oplus
\wt{\Nil}_{*-1}(R[H];R[G_1-H],R[G_2-H])
\end{equation}
and
\begin{equation}\label{eqn:Ktwist}
K_*(R[\ol{G}])~=~K_*(1-\alpha:\,R[H] \to R[H]) \oplus
\wt{\Nil}_{*-1}(R[H],\alpha)\oplus\wt{\Nil}_{*-1}(R[H],\alpha\inv)~.
\end{equation}
We establish isomorphisms
\[
\wt{\Nil}_*(R[H];R[G_1-H],R[G_2-H]) ~\cong~
\wt{\Nil}_*(R[H],\alpha)~\cong~\wt{\Nil}_*(R[H],\alpha\inv).
\]
\indent A homotopy equivalence $f:\,M \to X$ of
finite CW-complexes is \textbf{split along} a subcomplex $Y \subset X$ if it is a
cellular map and the restriction $f\vert:N=f\inv(Y) \to Y$ is also
a homotopy equivalence.  The $\wt{\Nil}_*$-groups arise as the
obstruction groups to splitting homotopy equivalences of finite
CW-complexes for codimension 1 $Y \subset X$ with injective
$\pi_1(Y) \to \pi_1(X)$, so that $\pi_1(X)$ is either an HNN
extension or an amalgamated free product (Farrell--Hsiang,
Waldhausen) --- see \fullref{codim1} for a brief review of the
codimension 1 splitting obstruction theory in the separating case
of an amalgamated free product.  In this paper we introduce the
considerably weaker notion of a homotopy equivalence in the
separating case being \textbf{semi-split} (\fullref{semisplit}). By way of geometric application we prove in
\fullref{Thm_TopSemisplit} that there is no obstruction to
topological semi-splitting in the finite-index case.

\subsection{Algebraic semi-splitting}

The following is a special case of our
main algebraic result (\ref{Thm_HigherDKR}, \ref{Thm_LowerDKR}) which shows that there is no obstruction to algebraic semi-splitting.

\begin{thm}\label{first}
Let $G$ be a group over $D_\infty$, with
$$H~=~G_1 \cap G_2~\subset~\ol{G}~=~H\rtimes_{\alpha}\Z \subset
G~=~G_1*_HG_2~.$$

\begin{enumerate}
\item
For any ring $R$ and $n \in \Z$ the
corresponding reduced $\Nil$-groups are isomorphic:
\[\wt{\Nil}_n(R[H];R[G_1- F],R[G_2- H])
~\cong~\wt{\Nil}_n(R[H],\alpha)~\cong~\wt{\Nil}_n(R[H],\alpha\inv)~.
\]

\item\label{first:partII}
The inclusion $\theta:\,R[\ol{G}]\to R[G]$ determines induction
and transfer maps
\[
\theta_!~:~K_n(R[\ol{G}]) \to K_n(R[G])~,~\theta^!~:~K_n(R[G]) \to
K_n(R[\ol{G}])~.
\]
For all integers $n \leqslant 1$, the $\wt{\Nil}_n(R[H],\alpha)$-$
\wt{\Nil}_n(R[H];R[G_1- H],R[G_2- H])$-components of the maps $\theta_!$ and $\theta^!$
in the decompositions \eqref{eqn:Ktwist} and \eqref{eqn:Kamalg}  are isomorphisms.
\end{enumerate}
\end{thm}

\begin{proof}
Part~(i) is a special case of \fullref{maink}.

Part~(ii) follows from \fullref{maink}, \fullref{lem:induction}, and \fullref{transfer}.
\end{proof}

The $n=0$ case will be discussed in more detail in \fullref{exp} and \fullref{First}.

\begin{rem} We do not seriously doubt that a more assiduous application
of higher $K$-theory would extend \fullref{first}(\ref{first:partII}) to all $n \in \Z$ (see also \cite{DQR}).
\end{rem}

As an application of \fullref{first}, we shall prove the following theorem. It shows that the Farrell--Jones Isomorphism Conjecture in algebraic $K$-theory can be reduced (up to dimension one) to the family of finite-by-cyclic groups, so that virtually cyclic groups of infinite dihedral type need not be considered.

\begin{thm}\label{thm:reduction}
Let $\g$ be any group, and let $R$ be any ring.  Then, for all integers $n < 1$, the following induced map
of equivariant homology groups, with coefficients in the algebraic
$K$-theory functor $\bK_R$, is an isomorphism:
\[
H^\g_n(E_{\fbc}\g;\bK_R) \longra H^\g_n(E_{\vc}\g;\bK_R) ~.
\]
Furthermore, this map is an epimorphism for $n=1$.
\end{thm}

In fact, this is a special case of a more general fibered version (\ref{isom_conj}).
\fullref{thm:reduction} has been proved for all degrees $n$ in \cite{DQR}; however our proof here uses only algebraic topology, avoiding the use of controlled topology.

The original reduced $\Nil$-groups $\wt{\Nil}_*(R)
=\wt{\Nil}_*(R,\mathrm{id})$ feature in the decompositions of Bass~\cite{Bass} and Quillen~\cite{Grayson}:
\begin{eqnarray*}
K_*(R[t]) &~=~& K_*(R) \oplus \wt{\Nil}_{*-1}(R)~,\\
K_*(R[\Z]) &~=~& K_*(R) \oplus K_{*-1}(R) \oplus \wt{\Nil}_{*-1}(R)\oplus \wt{\Nil}_{*-1}(R)~.
\end{eqnarray*}
In \fullref{Applications} we shall compute several examples which require \fullref{first}:
\begin{eqnarray*}
K_*(R[\Z_2 * \Z_2]) &~=~& \frac{K_*(R[\Z_2]) \oplus K_*(R[\Z_2])}{K_*(R)} \oplus \wt{\Nil}_{*-1}(R)\\
K_*(R[\Z_2 * \Z_3]) &~=~& \frac{K_*(R[\Z_2]) \oplus K_*(R[\Z_3])}{K_*(R)}  \oplus \wt{\Nil}_{*-1}(R)^{\infty}\\
\Wh(G_0 \x \Z_2 *_{G_0} G_0 \x \Z_2) &~=~& \frac{\Wh(G_0 \x \Z_2) \oplus \Wh(G_0 \x \Z_2) }{\Wh(G_0) } \oplus \wt{\Nil}_0(\Z[G_0])
\end{eqnarray*}
where $G_0 = \Z_2 \x \Z_2 \x \Z$.  The point here is that
$\wt{\Nil}_0(\Z[G_0])$ is an infinite torsion abelian group.
This provides the first example (\ref{Exm_NonzeroWaldNil})
of a non-zero $\wt{\Nil}$-group in the amalgamated product case and hence
the first example of a non-zero obstruction to splitting a homotopy equivalence in the separating case $(A)$.

\subsection{The Nil-Nil Theorem}\label{exp}

We establish isomorphisms between two types of
codimension 1 splitting obstruction nilpotent class groups, for any ring $R$. The
first type, for separated splitting, arises in the decompositions of
the algebraic $K$-theory of the $R$-coefficient group ring $R[G]$ of
a group $G$ over $D_\infty$, with an epimorphism $p:\,G\to D_{\infty}$ onto the
infinite dihedral group $D_{\infty}$. The second type, for
non-separated splitting, arises from the $\alpha$-twisted polynomial
ring $R[H]_{\alpha}[t]$, with $H=\mathrm{ker}(p)$ and $\alpha:\,F \to F$
an automorphism such that
\[
\ol{G}~=~\mathrm{ker}(\pi\circ p:\,G \to \Z_2)~=~H\rtimes_{\alpha}\Z
\]
where $\pi:\,D_{\infty} \to \Z_2$ is the unique epimorphism with
infinite cyclic kernel. Note:
\begin{itemize}
\item[(A)] $D_{\infty}=\Z_2*\Z_2$ is the free product of two
    cyclic groups of order 2, whose generators will be denoted
    $t_1,t_2$.
\item[(B)] $D_{\infty} = \langle ~t_1, t_2 ~|~ t_1^2 = 1 = t_2^2~ \rangle$ contains the infinite cyclic group $\Z
    =\langle t \rangle$ as a subgroup of index 2 with
    $t=t_1t_2$. In fact there is a short exact sequence with a
    split epimorphism
$$\xymatrix{\{1\} \ar[r] &  \Z \ar[r] &
D_{\infty} \ar[r]^-{\di{\pi}} & \Z_2 \ar[r]& \{1\}~.}$$
\end{itemize}
More generally, if $G$ is a group over $D_\infty$, with an epimorphism $p:\,G \to
D_{\infty}$, then:
\begin{itemize}
\item[(A)] $G=G_1*_HG_2$ is a free product with amalgamation
    of two groups
    \[
    G_1~=~\mathrm{ker}(p_1:\,G \to \Z_2)~,~G_2~=~\mathrm{ker}(p_2:\,G \to \Z_2) \subset G
    \]
    amalgamated over their common subgroup $H=\mathrm{ker}(p)=G_1
\cap G_2$ of index 2 in both $G_1$ and $G_2$.
\item[(B)] $G$ has a subgroup $\ol{G} = \ker(\pi \circ p:\, G
    \to \Z_2)$ of index 2 which is an HNN extension $\ol{G}
    =H\rtimes_{\alpha} \Z$ where $\alpha:\,H \to H$ is
    conjugation by an element $t \in \ol{G}$ with $p(t)=t_1t_2 \in D_{\infty}$.
\end{itemize}

\subsubsection*{The $K$-theory of type (A)}

For any ring $S$ and $S$-bimodules $\fB_1,\fB_2$, we write the $S$-bimodule $\fB_1\otimes_S\fB_2$ as $\fB_1\fB_2$, and we suppress left-tensor products of maps with the identities $\mathrm{id}_{\fB_1}$ or $\mathrm{id}_{\fB_2}$.  The exact category
$\NIL(S;\fB_1,\fB_2)$ has objects being quadruples
$(P_1,P_2,\rho_1,\rho_2)$ consisting of finitely generated  (= finitely generated) projective $S$-modules $P_1,P_2$ and $S$-module morphisms
\[
\rho_1~:~P_1 \longrightarrow \fB_1 P_2~,~\rho_2~:~P_2 \longrightarrow \fB_2P_1
\]
such that $\rho_2\rho_1:\,P_1 \to \fB_1 \fB_2P_1$ is \textbf{nilpotent} in the sense that
\[
(\rho_2 \circ \rho_1)^k~=0~:P_1 \longrightarrow (\fB_1 \fB_2)(\fB_1 \fB_2) \cdots (\fB_1 \fB_2)P_1
\]
for some $k \geqslant 0$.
The morphisms are pairs $(f_1:\, P_1 \to P'_1,f_2:\, P_2 \to P'_2)$ such that $f_2 \circ \rho_1 = \rho'_1 \circ f_1$ and $f_1 \circ \rho_2 = \rho'_2 \circ f_2$.  Recall the Waldhausen $\Nil$-groups $\Nil_*(S;\fB_1,\fB_2) := K_*(\NIL(S;\fB_1,\fB_2))$, and the reduced $\Nil$-groups $\wt{\Nil}_*$ satisfy
\[
\Nil_*(S;\fB_1,\fB_2)~=~K_*(S)\oplus K_*(S) \oplus
\wt{\Nil}_*(S;\fB_1,\fB_2)~.
\]
An object $(P_1,P_2,\rho_1,\rho_2)$ in $\NIL(S;\fB_1,\fB_2)$ is \textbf{semi-split} if the $S$-module isomorphism $\rho_2:P_2 \to \fB_2P_1$ is an isomorphism.

Let $R$ be a ring which is an amalgamated free product
$$R~=~R_1*_SR_2$$
with $R_k=S \oplus \fB_k$ for $S$-bimodules $\fB_k$ which are free $S$-modules, $k=1,2$. The algebraic $K$-groups were shown in \cite{Waldhausen_1969,Waldhausen_1973,Waldhausen_Rings} to fit into a long exact sequence
$$\begin{array}{l}
\dots \to K_n(S)\oplus \wt{\Nil}_n(S;\fB_1,\fB_2) \to K_n(R_1)\oplus K_n(R_2)\to K_n(R)\\[1ex]
\hskip150pt
\to K_{n-1}(S) \oplus \wt{\Nil}_{n-1}(S;\fB_1,\fB_2) \to \dots
\end{array}$$
with $K_n(R) \to \wt{\Nil}_{n-1}(S;\fB_1,\fB_2)$ a split surjection.

 For any ring $R$ a based finitely generated  free $R$-module chain complex $C$ has a torsion
$\tau(C) \in K_1(R)$. The torsion of a chain equivalence $f:C \to D$ of based
finitely generated  free $R$-module chain complexes is the torsion of the algebraic mapping cone
$$\tau(f)~=~\tau(\fC(f)) \in K_1(R)~.$$
By definition, the chain equivalence is \textbf{simple} if $\tau(f)=0 \in K_1(R)$.

For $R=R_1*_SR_2$ the algebraic analogue of codimension 1 manifold
transversality shows that every based  finitely generated  free $R$-module
chain complex $C$ admits a Mayer--Vietoris presentation
$$\fC~:~0 \to R\otimes_SD \to (R\otimes_{R_1}C_1) \oplus (R\otimes_{R_2}C_2) \to C \to 0$$
with $C_k$ a based finitely generated  free $R_k$-module chain
complex, $D$ a based finitely generated  free $S$-module chain
complex with $R_k$-module chain maps $R_k\otimes_SD \to C_k$, and
$\tau(\fC)=0 \in K_1(R)$. This was first proved in
\cite{Waldhausen_1969,Waldhausen_1973}; see also \cite[Remark
8.7]{Ranicki_Novikov} and \cite{Ranicki_Trans}. A contractible $C$
\textbf{splits} if it is simple chain equivalent to a chain complex
(also denoted by $C$) with a Mayer--Vietoris presentation $\fC$
with $D$ contractible, in which case $C_1,C_2$ are also
contractible and the  torsion $\tau(C) \in K_1(R)$ is such that
$$\tau(C)~=~\tau(R\otimes_{R_1}C_1)+\tau(R\otimes_{R_2}C_2)-
\tau(R\otimes_SD) \in \mathrm{im}(K_1(R_1)\oplus K_1(R_2) \to K_1(R))~.$$
By the algebraic obstruction theory of \cite{Waldhausen_1969} $C$ splits if and only if
$$\tau(C) \in \mathrm{im}(K_1(R_1)\oplus K_1(R_2) \to K_1(R))~=~\mathrm{ker}(K_1(R) \to K_0(S)\oplus \wt{\Nil}_0(S;\fB_1,\fB_2))~.$$

For any ring $R$ the group ring $R[G]$ of an amalgamated free product of groups $G=G_1*_HG_2$ is an amalgamated free product of rings
$$R[G]~=~R[G_1]*_{R[H]}R[G_2]~.$$
If $H \to G_1$, $H \to G_2$ are injective then the $R[H]$-bimodules
$R[G_1-H]$, $R[G_2-H]$ are free, and Waldhausen~\cite{Waldhausen_Rings} decomposed the algebraic $K$-theory of $R[G]$ as
\[
K_*(R[G])~=~K_*(R[H] \to R[G_1]\times R[G_2]) \oplus \wt{\Nil}_{*-1}(R[F];R[G_1-H],R[G_2- H])~.
\]
In particular, there is defined a split monomorphism
\[
\sigma_A~:~\wt{\Nil}_{*-1}(R[H];R[G_1- H],R[G_2- H])
\longrightarrow K_*(R[G]) ~,
\]
which for $*=1$ is given by
\begin{align*}
\sigma_A ~:~ &\wt{\Nil}_0(R[H];R[G_1- H],R[G_2- H])
\longrightarrow K_1(R[G])~;\\
& [P_1,P_2,\rho_1,\rho_2] \longmapsto \left[ R[G]\otimes_{R[H]}(P_1 \oplus P_2), \begin{pmatrix} 1 & \rho_2 \\ \rho_1 & 1\end{pmatrix} \right]~.
\end{align*}

\subsubsection*{The $K$-theory of type (B)}

Given a ring $R$ and an $R$-bimodule $\fB$, consider the tensor algebra $T_R(\fB)$ of $\fB$ over $R$:
\[
T_R(\fB)~:=~R \;\oplus\; \fB \;\oplus\; \fB\fB \;\oplus\; \cdots.
\]
The $\Nil$-groups $\Nil_*(R;\fB)$ are defined to be the algebraic
$K$-groups $K_*(\NIL(R;\fB))$ of the exact category $\NIL(R;\fB)$
with objects pairs $(P,\rho)$ with $P$ a finitely generated
projective $R$-module and $\rho:\,P \to \fB P$ an $R$-module
morphism, \textbf{nilpotent} in the sense that for some $k \geqslant
0$, we have
\[
\rho^k = 0 ~:~ P \longrightarrow \fB  P \longrightarrow \cdots \longrightarrow \fB^k  P.
\]
The reduced $\Nil$-groups $\wt{\Nil}_*$ are such that
\[\Nil_*(R;\fB)~=~K_*(R) \oplus \wt{\Nil}_*(R;\fB)~.\]
Waldhausen~\cite{Waldhausen_Rings} proved that if $\fB$ is finitely generated  projective as a
left $R$-module and free as a right $R$-module, then
\[
K_*(T_R(\fB))~=~K_*(R) \oplus \wt{\Nil}_{*-1}(R;\fB).
\]
There is a split monomorphism
\[
\sigma_B~:~\wt{\Nil}_{*-1}(R;\fB) \longrightarrow K_*(T_R(\fB))
\]
which for $*=1$ is given by
\[
\sigma_B~:~\wt{\Nil}_0(R;\fB) \longrightarrow K_1(T_R(\fB))~;~
[P,\rho] \longmapsto [T_R(\fB)P,1-\rho]~.
\]
In particular, for $\fB=R$, we have:
\begin{align*}
&\Nil_*(R;R)~=~\Nil_*(R)~,~\wt{\Nil}_*(R;R)~=~\wt{\Nil}_*(R)\\
&T_R(\fB)~=~R[t]~,~K_*(R[t])~=~K_*(R) \oplus \wt{\Nil}_{*-1}(R)~.
\end{align*}

\subsubsection*{Relating the $K$-theory of types (A) and (B)}

Recall that a small category $I$ is \textbf{filtered} if:
\begin{itemize}
\item
for any pair of objects $\alpha, \alpha'$ in $I$, there exist an
object $\beta$ and morphisms $\alpha \to \beta$ and $\alpha' \to
\beta$ in $I$, and

\item
for any pair of morphisms $u, v:\, \alpha \to \alpha'$
in $I$, there exists an object $\beta$ and morphism $w:\, \alpha'
\to \beta$ such that $w \circ u = w \circ v$.
\end{itemize}
Note that any directed poset $I$ is a filtered category.  A \textbf{filtered colimit} is a colimit over a filtered category.

\begin{thm}[The Nil-Nil Theorem]\label{maink}
Let $R$ be a ring. Let $\fB_1,\fB_2$ be $R$-bimodules.  Suppose that
 $\fB_2 = \colim_{\alpha \in I}
\fB_2^\alpha$ is a filtered colimit limit of $R$-bimodules such that each
$\fB_2^\alpha$ is  a finitely generated   projective left $R$-module.
Then, for all $n \in \Z$, the $\Nil$-groups of the triple
$(R;\fB_1,\fB_2)$ are related to the $\Nil$-groups of the pair
$(R;\fB_1 \fB_2)$ by isomorphisms
\begin{eqnarray*}
\Nil_n(R;\fB_1,\fB_2) &\cong &\Nil_n(R;\fB_1 \fB_2) \oplus  K_n(R)~,\\
\wt{\Nil}_n(R;\fB_1,\fB_2) &\cong& \wt{\Nil}_n(R;\fB_1 \fB_2)~.
\end{eqnarray*}
In particular, for $n=0$ and  $\fB_2$  a finitely generated  projective left $R$-module, there are defined inverse isomorphisms
\begin{align*}
&i_*~:~\Nil_0(R;\fB_1 \fB_2)\oplus K_0(R)
\xrightarrow{~\cong~} \Nil_0(R;\fB_1,\fB_2)~;\\
&([P_1,\rho_{12}:\,P_1 \to \fB_1 \fB_2P_1],[P_2])
\longmapsto [P_1, \fB_2  P_1\oplus P_2,\begin{pmatrix}
\rho_{12}\\ 0\end{pmatrix},(1~0)]~,\\
&j_*~:~\Nil_0(R;\fB_1,\fB_2) \xrightarrow{~\cong~}
\Nil_0(R;\fB_1   \fB_2)\oplus K_0(R)~;\\
&[P_1,P_2,\rho_1:\,P_1 \to \fB_1 P_2,\rho_2:
P_2 \to \fB_2P_1]\longmapsto ([P_1,\rho_2\circ \rho_1],
[P_2] - [\fB_2P_1])~.
\end{align*}
The reduced versions are the inverse isomorphisms
\begin{align*}
&i_*~:~\wt{\Nil}_0(R;\fB_1 \fB_2) \xrightarrow{~\cong~} \wt{\Nil}_0(R;\fB_1,\fB_2)~;~
[P_1,\rho_{12}] \longmapsto [P_1, \fB_2  P_1,\rho_{12},1]~,\\
&j_*~:~\wt{\Nil}_0(R;\fB_1,\fB_2) \xrightarrow{~\cong~}
\wt{\Nil}_0(R;\fB_1   \fB_2)~;~
[P_1,P_2,\rho_1,\rho_2]\longmapsto [P_1,\rho_2\circ \rho_1]
\end{align*}
with $i_*(P_1,\rho_{12})=(P_1, \fB_2  P_1,\rho_{12},1)$ semi-split.
\end{thm}

\begin{proof}
This follows immediately from \fullref{Thm_HigherDKR} and \fullref{Thm_LowerDKR}.
\end{proof}

\begin{rem}
\fullref{maink} was already known to Pierre Vogel in 1990---see \cite{Vogel}.
\end{rem}

\section{Higher $\Nil$-groups}\label{nilsection}

In this section, we shall prove \fullref{maink} for
non-negative degrees.

Quillen~\cite{Quillen} defined the $K$-theory space $K\fE := \Omega
BQ(\fE)$ of an exact category $\fE$.  The space $BQ(\fE)$ is the
geometric realization of the simplicial set $N_\bullet Q(\fE)$,
which is the nerve of a certain category $Q(\fE)$ associated to
$\fE$. The algebraic $K$-groups of $\fE$ are defined for $* \in \Z$
$$K_*(\fE)~:=~\pi_*(K\fE)$$
using a nonconnective delooping for $* \leqslant -1$.
In particular, the algebraic $K$-groups of a ring $R$ are the
algebraic $K$-groups
$$K_*(R)~:=~K_*(\PROJ(R))$$
of the exact category $\PROJ(R)$ of finitely generated  projective $R$-modules.
The $\NIL$-categories defined in the Introduction all
have the structure of exact categories.

\begin{thm}\label{Thm_HigherDKR}Let $\fB_1$ and  $\fB_2$ be bimodules over a ring $R$.  Let $j$ be the exact functor
\[
j:\, \NIL(R;\fB_1,\fB_2) \longra \NIL(R;\fB_1  \fB_2)~;~ (P_1,P_2,\rho_1,\rho_2) \longmapsto (P_1,  \rho_2 \circ \rho_1) ~.
\]
\begin{enumerate}

\item  If $\fB_2$ is finitely generated  projective as a left $R$-module, then there is an exact functor
$$
i:\, \NIL(R;\fB_1  \fB_2) \longra \NIL(R;\fB_1,\fB_2) ~;~ (P,\rho) \longmapsto (P,\fB_2 P,\rho, 1)
$$
such that $i(P,\rho)=(P,\fB_2P,\rho,1)$ is semi-split, $j \circ i = 1$,  and  $i_*$ and $j_*$ induce inverse isomorphisms on the reduced Nil-groups
$$
\wt \Nil_*(R; \fB_1 \fB_2) ~\cong~  \wt \Nil_*(R; \fB_1, \fB_2)~.
$$

\item  If $\fB_2=\colim_{\alpha \in I}
\fB_2^\alpha$ is a filtered colimit of bimodules each of which is finitely generated  projective as a left $R$-module, then there is a unique exact functor $i$ so that the following diagram commutes for all $\alpha \in I$
$$
\begin{diagram}
\node{\NIL(R;\fB_1  \fB_2)} \arrow{e,t}{\di{i}} \node{\NIL(R;\fB_1,  \fB_2)} \\
\node{\NIL(R;\fB_1  \fB_2^\alpha)} \arrow{e,t}{\di{i^\alpha}} \arrow{n} \node{\NIL(R;\fB_1,  \fB_2^\alpha)~.} \arrow{n}
\end{diagram}
$$
Then $j \circ i = 1$ and  $i_*$ and $j_*$ induce inverse isomorphisms on the reduced Nil-groups
$$
\wt \Nil_*(R; \fB_1 \fB_2) ~\cong~  \wt \Nil_*(R; \fB_1, \fB_2)~.
$$

\end{enumerate}

\end{thm}

\begin{proof}
\noindent (1)  Note that there are split injections of exact categories
\begin{gather*}
\PROJ(R) \times \PROJ(R) \to \NIL(R; \fB_1, \fB_2)~;~(P_1,P_2) \longmapsto (P_1,P_2,0,0)~,\\
\PROJ(R)  \to \NIL(R; \fB_1 \fB_2)~;~(P) \longmapsto (P,0)~,
\end{gather*}
which underly the definition of the reduced Nil groups.
Since both $i$ and $j$ take the image of the split injection to the image of the other split injection, they induce maps $i_*$ and $j_*$ on the reduced Nil groups.  Since $j \circ i = 1$, it follows that $j_* \circ i_* = 1$.

In preparation for the proof that $i_* \circ j_* = 1$, consider the following objects of  $\NIL(R;\fB_1,\fB_2)$
\begin{eqnarray*}
x &:=& (P_1, P_2, \rho_1, \rho_2)\\
x' &:=& (P_1, \fB_2 P_1 \oplus P_2, \begin{pmatrix}0\\ \rho_1\end{pmatrix},
\begin{pmatrix}1 & \rho_2\end{pmatrix})\\
x'' &:=& (P_1, \fB_2 P_1, \rho_2 \circ \rho_1, 1)\\
a &:=& (0,P_2,0,0)\\
a' &:=& (0,\fB_2 P_1,0,0)
\end{eqnarray*}
with $x''$ semi-split. Note that $(i \circ j)(x) = x''$.
Define morphisms
\begin{eqnarray*}
f &:=& (1, \begin{pmatrix}0\\ 1\end{pmatrix}) :x \longra x'\\
f' &:=& (1, \begin{pmatrix}1 & \rho_2\end{pmatrix}) :x' \longra x''\\
g &:=& (0, \begin{pmatrix}-\rho_2 \\ 1\end{pmatrix}): a \longra x'\\
g' &:=& (0, \begin{pmatrix}1 & 0\end{pmatrix}) :x' \longra a'\\
h &:=& (0, \rho_2) :a \longra a'.
\end{eqnarray*}

There are exact sequences
\begin{gather*}
\begin{CD}
0 @>>> x \oplus a @>{\begin{pmatrix}f & g\\ 0 & 1\end{pmatrix}}>> x' \oplus a
@>{\begin{pmatrix}g' & h\end{pmatrix}}>> a' @>>> 0
\end{CD}\\
\begin{CD}
0 @>>> a @>{\di{g}}>> x' @>{\di{f'}}>> x'' @>>> 0.
\end{CD}
\end{gather*}
Define exact functors $F', F'', G, G'~:~ \NIL(R;\fB_1,\fB_2) \to \NIL(R;\fB_1,\fB_2)$ by
\[
F'(x) = x', \quad F''(x) = x'', \quad G(x) = a, \quad G'(x) = a'.
\]
Thus we have two exact sequences of exact functors
\begin{gather*}
\begin{CD}
0 @>>> 1 \oplus G @>>> F' \oplus G @>>> G' @>>> 0
\end{CD}\\
\begin{CD}
0 @>>> G @>>> F' @>>> F'' @>>> 0.
\end{CD}
\end{gather*}
Recall $j \circ i = 1$, and note $i \circ j = F''$.  By Quillen's
Additivity Theorem~\cite[page 98, Corollary 1]{Quillen}, we obtain
homotopies $KF' \simeq 1 + KG'$ and $KF' \simeq KG + KF''$. Then
$$Ki \circ Kj~=~KF'' \simeq 1 + (KG'-KG)~,$$
where the subtraction uses the loop space structure.
Observe that both $G$ and $G'$ send $\NIL(R;\fB_1,\fB_2)$ to the image of
$\PROJ(R) \times \PROJ(R)$.  Thus $i_* \circ j_* = 1$ as desired.

\noindent (2)
It is straightforward to show that tensor product commutes with
colimits over a category. Moreover, for any object $x =
(P_1,P_2,\rho_1:\, P_1 \to \fB_1 P_2, \rho_2:\, P_2 \to \fB_2 P_1)$,
since $P_2$ is finitely generated , there exists $\alpha \in I$ such
that $\rho_2$ factors through a map $P_2 \to \fB_2^\alpha P_1$, and
similarly for short exact sequences of nil-objects. We thus obtain induced
isomorphisms of exact categories:
\begin{eqnarray*}
\colim_{\alpha \in I} \NIL(R;\fB_1  \fB_2^\alpha) &\longrightarrow& \NIL(R;\fB_1  \fB_2)\\
\colim_{\alpha \in I} \NIL(R;\fB_1,\fB_2^\alpha) &\longrightarrow& \NIL(R;\fB_1,\fB_2).
\end{eqnarray*}
This justifies the existence and uniqueness of the functor $i$.

By Quillen's colimit observation \cite[Section 2, Equation (9), page
20]{Quillen}, we obtain induced weak homotopy equivalences of
$K$-theory spaces:
\begin{eqnarray*}
\colim_{\alpha \in I} K\NIL(R;\fB_1  \fB_2^\alpha) &\longrightarrow& K\NIL(R;\fB_1  \fB_2)\\
\colim_{\alpha \in I}
K\NIL(R;\fB_1,\fB_2^\alpha) &\longrightarrow& K\NIL(R;\fB_1,\fB_2).
\end{eqnarray*}
The remaining assertions of part (2) then follow from part (1).
\end{proof}

\begin{rem}
The proof of \fullref{Thm_HigherDKR} is best understood in
terms of finite chain complexes  $x=(P_1,P_2,\rho_1,\rho_2)$ in
the category $\NIL(R;\fB_1,\fB_2)$, assuming that $\fB_2$ is
a finitely generated  projective left $R$-module. Any such $x$ represents a
class
\[
[x]~=~\sum\limits^{\infty}_{r=0}(-1)^r[(P_1)_r,(P_2)_r,\rho_1,\rho_2]
\in \Nil_0(R;\fB_1,\fB_2)~.
\]
The key observation is that $x$ determines a finite chain complex
$x'=(P'_1,P'_2,\rho'_1,\rho'_2)$ in $\NIL(R;\fB_1,\fB_2)$
which is semi-split in the sense that  $\rho'_2:\,P'_2\to \fB_2P'_1$ is a chain equivalence, and such that
\begin{equation}\label{eqn:projEulerclass}
[x]~=~[x'] \in \wt{\Nil}_0(R;\fB_1,\fB_2)~.
\end{equation}
Specifically, let $P'_1=P_1$, $P'_2=\fM(\rho_2)$, the algebraic
mapping cylinder of the chain map $\rho_2:\,P_2 \to
\fB_2P_1$, and let
\[\begin{array}{rcccccl}
\rho'_1 &=&
\begin{pmatrix} 0 \\ 0 \\ \rho_1 \end{pmatrix} &:&
P'_1=P_1 &\longrightarrow& \fB_1 P'_2=\fM(1_{\fB_1} \otimes \rho_2)~,\\[2ex]
\rho'_2 &=& \begin{pmatrix} 1 & 0 & \rho_2\end{pmatrix} &:& P'_2~=~\fM(\rho_2) &\longrightarrow& \fB_2P_1~,
\end{array}\]
so that $P'_2/P_2=\fC(\rho_2)$ is the algebraic mapping cone of
$\rho_2$. Moreover, the proof of \eqref{eqn:projEulerclass} is sufficiently functorial
to establish not only that the following maps of the reduced
nilpotent class groups are inverse isomorphisms:
\[\begin{array}{l}
i~:~\wt{\Nil}_0(R;\fB_1 \fB_2) \longrightarrow
\wt{\Nil}_0(R;\fB_1,\fB_2) ~;~(P,\rho) \longmapsto (P,\fB_2P,\rho,1)~,\\[1ex]
j~:~\wt{\Nil}_0(R;\fB_1,\fB_2) \longrightarrow \wt{\Nil}_0(R;\fB_1 \fB_2)~;~
[x] \longmapsto [x']~,
\end{array}
\]
but also that there exist isomorphisms of $\wt{\Nil}_n$ for
all higher dimensions $n > 0$, as shown above.  In order to prove
equation \eqref{eqn:projEulerclass}, note that $x$ fits into the sequence
\begin{equation}\label{eqn:mappingcone}
\xymatrix@C+10pt{0 \ar[r] & x \ar[r]^-{\di{(1,u)}}& x'\ar[r]^-{\di{(0,v)}}&y
\ar[r]& 0}
\end{equation}
with
\[\begin{array}{rcl}
y &=&(0,\fC(\rho_2),0,0)~,\\[2ex]
u &=& \begin{pmatrix} 0 \\ 0 \\ 1 \end{pmatrix}~:~P_2
\to P'_2~=~\fM(\rho_2)~,\\[3ex]
v &=& \begin{pmatrix} 1 & 0 & 0 \\ 0 & 1 & 0 \end{pmatrix}~:~
P'_2 = \fM(\rho_2) \to  \fC(\rho_2)
\end{array}\]
and
\[[y]~=~\sum\limits^{\infty}_{r=0}(-)^r[0,(\fB_2P_1)_{r-1}\oplus(P_2)_r,0,0]~=~0
\in \wt{\Nil}_0(R;\fB_1,\fB_2)~.\] The projection
$\fM(\rho_2) \to \fB_2P_1$ defines a chain equivalence \[x'
~\simeq~ (P_1,\fB_2P_1,\rho_2\circ \rho_1,1)~=~ij(x)\] so
that
\[[x]~=~[x']- [y]~=~[P_1,\fB_2P_1,\rho_2\circ \rho_1,1]~
=~ij[x] \in \wt{\Nil}_0(R;\fB_1,\fB_2)~.\]
Now suppose that $x$ is a 0-dimensional chain complex in
$\NIL(R;\fB_1,\fB_2)$, that is, an object as in the proof
of \fullref{Thm_HigherDKR}. Let $x',x'',a,a',f,f',g,g',h$ be
as defined there. The exact sequence of \eqref{eqn:mappingcone} can be written as
the short exact sequence of chain complexes
\[\xymatrix{ & & a \ar[d]^-{\di{g}} \ar@{=}[r] & a \ar[d]^-{\di{-h}} & \\
0 \ar[r] & x \ar[r]^-{\di{f}} & x' \ar[r]^-{\di{g'}} & a' \ar[r] &
0~.}\]
The first  exact sequence of the proof of \fullref{Thm_HigherDKR} is now immediate:
\[\xymatrix@C+15pt{0 \ar[r] &x \oplus a \ar[r]^-{\begin{pmatrix} f  & g \\
0 & 1 \end{pmatrix}} & x'\oplus a \ar[r]^-{\begin{pmatrix} g' & h
\end{pmatrix}} & a' \ar[r] &0~.}\]
The second  exact sequence is self-evident:
\[\xymatrix{0 \ar[r] &a \ar[r]^-{\di{g}} & x' \ar[r]^-{\di{f'}} & x'' \ar[r]
&0~.}\]
\end{rem}

\section{Lower $\Nil$-groups}\label{Sec_Lower}

\subsection{Cone and suspension rings}

Let us recall some additional structures on the tensor product of
modules.

Originating from ideas of Karoubi--Villamayor~\cite{KaroubiVillamayor}, the following concept was studied
independently by S\,M Gersten~\cite{Gersten} and J\,B Wagoner~\cite{Wagoner} in the construction of the non-connective
$K$-theory spectrum of a ring.

\begin{dfn}[Gersten, Wagoner]
The \textbf{cone ring} $\Lambda\Z$ is the subring of $\omega\times\omega$-matrices over $\Z$ such that each row and column have only a
finite number of non-zero entries. The \textbf{suspension ring} $\Sigma\Z$ is the quotient ring of $\Lambda\Z$ by the
two-sided ideal of matrices with only a finite number of non-zero
entries. For each $n \in \N$, define the rings
\[
\Sigma^n\Z ~:=~ \underbrace{\Sigma\Z \xo_\Z \cdots \xo_\Z \Sigma\Z}_{n \text{
copies}} \quad\text{with}\quad \Sigma^0\Z = \Z.
\]
For a ring $R$ and for $n \in \N$, define the ring
$\Sigma^n R := \Sigma^n\Z \xo_\Z R$.
\end{dfn}

Roughly speaking, the suspension should be regarded
as the ring of ``bounded modulo compact operators.''  Gersten and
Wagoner showed that $K_i(\Sigma^n R)$ is naturally isomorphic to
$K_{i-n}(R)$ for all $i, n \in \Z$, in the sense of Quillen when
the subscript is positive, in the sense of Grothendieck when the
subscript is zero, and in the sense of Bass when the subscript is
negative.

For an $R$-bimodule $\fB$, define the $\Sigma^n
R$-bimodule $\Sigma^n \fB := \Sigma^n\Z \xo_\Z \fB$.

\begin{lem}\label{Lem_Transpose}
Let $R$ be a ring.  Let $\fB_1, \fB_2$ be $R$-bimodules. Then, for
each $n \in \N$, there is a natural isomorphism of $\Sigma^n
R$-bimodules:
\[
t_n:\, \Sigma^n(\fB_1  \fB_2) \longra \Sigma^n \fB_1 \xo_{\Sigma^n R} \Sigma^n \fB_2~;~
s \xo (b_1 \xo b_2) \longmapsto (s \xo b_1) \xo (1_{\Sigma^n \Z} \xo b_2).
\]
\end{lem}

\begin{proof}
By transposition of the middle two factors, note that
\[
\Sigma^n \fB_1 \xo_{\Sigma^n R} \Sigma^n \fB_2 = (\Sigma^n \Z \xo_\Z \fB_1)
\xo_{(\Sigma^n\Z \xo_\Z R)} (\Sigma^n \Z \xo_\Z \fB_2)
\]
is isomorphic to
\[
(\Sigma^n\Z \xo_{\Sigma^n\Z} \Sigma^n\Z) \xo_{\Z}
(\fB_1  \fB_2) = \Sigma^n\Z \xo_\Z (\fB_1  \fB_2) = \Sigma^n(\fB_1  \fB_2).
\proved\]
\end{proof}

\subsection{Definition of lower $\Nil$-groups}

\begin{dfn}\label{Defn_NilTensorAlgebra}
Let $R$ be a ring. Let $\fB$ be an $R$-bimodule.  For all $n \in
\N$, define
\begin{eqnarray*}
\Nil_{-n}(R;\fB) &:=& \Nil_0(\Sigma^nR;\Sigma^n \fB)\\
\wt{\Nil}_{-n}(R;\fB) &:=& \wt \Nil_0(\Sigma^nR;\Sigma^n \fB).
\end{eqnarray*}
\end{dfn}

\begin{dfn}\label{Defn_NilAmalgam}
Let $R$ be a ring. Let $\fB_1, \fB_2$ be $R$-bimodules.  For all
$n \in \N$, define
\begin{eqnarray*}
\Nil_{-n}(R;\fB_1,\fB_2) &:=& \Nil_0(\Sigma^nR;\Sigma^n\fB_1,\Sigma^n\fB_2)\\
\wt{\Nil}_{-n}(R;\fB_1,\fB_2) &:=&  \wt \Nil_0(\Sigma^nR;\Sigma^n\fB_1,\Sigma^n\fB_2).
\end{eqnarray*}
\end{dfn}

The next two theorems follow from the definitions and
\cite[Theorems 1,3]{Waldhausen_Rings}.

\begin{thm}[Waldhausen]\label{Thm_WaldhausenTensorRing}
Let $R$ be a ring and $\fB$ be an $R$-bimodule. Consider the tensor
ring
\[
 T_R(\fB) := R \oplus \fB \oplus \fB^2 \oplus \fB^3\oplus \cdots~.
\]
Suppose $\fB$ is finitely generated  projective as a left $R$-module
and free as a right $R$-module. Then, for all $n \in \N$, there is a
split monomorphism
\[
\sigma_B~:~\wt{\Nil}_{-n}(R;\fB) \longrightarrow K_{1-n}( T_R(\fB))
\]
given for $n=0$ by the map
\[
\sigma_B~:~\Nil_0(R;\fB) \longrightarrow K_1( T_R(\fB))~;~
[P,\rho] \longmapsto \left[~ T_R(\fB) P, 1-\widehat{\rho} ~\right]~,
\]
where $\widehat{\rho}$ is defined using $\rho$ and
multiplication in $ T_R(\fB)$.

Furthermore, there is a natural decomposition
\[
K_{1-n}( T_R(\fB)) = K_{1-n}(R) \oplus \wt{\Nil}_{-n}(R;\fB).
\]
\end{thm}
For example, the last assertion of the theorem follows from the
equations:
\begin{align*}
K_{1-n}(T_R(\fB)) & = K_1(\Sigma^n T_R(\fB)) \\
& = K_1(T_{\Sigma^nR}(\Sigma^n\fB)) \\
& = K_1(\Sigma^n R) \oplus \wt{\Nil}_0(\Sigma^nR; \Sigma^n \fB) \\
& = K_{1-n}(R) \oplus \wt{\Nil}_{-n}(R; \fB).
\end{align*}

\begin{thm}[Waldhausen]\label{Waldhausen}
Let $R, A_1, A_2$ be rings. Let $R \to A_i$ be ring monomorphisms
such that $A_i = R \oplus \fB_i$ for $R$-bimodules $\fB_i$.
Consider the pushout of rings
\begin{multline*}
A = A_1 *_R A_2 = R \oplus (\fB_1 \oplus \fB_2) \oplus
(\fB_1 \fB_2 \oplus \fB_2   \fB_1) \\ \oplus (\fB_1
\fB_2 \fB_1 \oplus \fB_2 \fB_1 \fB_2)
\oplus \cdots.
\end{multline*}
Suppose each $\fB_i$ is free as a right $R$-module. Then, for all
$n \in \N$, there is a split monomorphism
\[
\sigma_A~:~\wt{\Nil}_{-n}(R;\fB_1,\fB_2) \longrightarrow K_{1-n}(A),
\]
given for $n=0$ by the map
\[
{\Nil}_0(R;\fB_1,\fB_2) \longrightarrow K_1(A)~;~
[P_1,P_2,\rho_1,\rho_2] \longmapsto \left[ (A  P_1)\oplus (A  P_2), \begin{pmatrix} 1 &
\widehat{\rho}_2\\ \widehat{\rho}_1 & 1\end{pmatrix} \right]~,
\]
where $\widehat{\rho}_i$ is  defined using $\rho_i$ and
multiplication in $A_i$ for $i = 1,2$.

Furthermore, there is a natural Mayer--Vietoris type exact
sequence
\[\begin{CD}
\cdots @>{\partial}>> K_{1-n}(R) @>>> K_{1-n}(A_1) \oplus K_{1-n}(A_2) @>>>\\
\dfrac{K_{1-n}(A)}{\wt{\Nil}_{-n}(R;\fB_1,\fB_2)}
@>{\partial}>> K_{-n}(R) @>>> \cdots
\end{CD}\]
\end{thm}

\subsection{The isomorphism for lower $\Nil$-groups}

\begin{thm}\label{Thm_LowerDKR}
Let $R$ be a ring.  Let $\fB_1, \fB_2$ be $R$-bimodules.  Suppose
that $\fB_2=\colim_{\alpha \in I}
\fB_2^\alpha$ is a filtered  colimit of $R$-bimodules $\fB_2^\alpha$,
each of which is a finitely generated  projective left $R$-module. Then, for all $n
\in \N$, there is an induced isomorphism:
\[
\Nil_{-n}(R;\fB_1  \fB_2) \oplus K_{-n}(R) \longra \Nil_{-n}(R;\fB_1,\fB_2).
\]
\end{thm}

\begin{proof}
Let $n \in \N$. By \fullref{Lem_Transpose} and \fullref{Thm_HigherDKR}, there are induced isomorphisms:
\begin{multline*}
\Nil_{-n}(R;\fB_1  \fB_2) \oplus K_{-n}(R) =
\Nil_0(\Sigma^nR;\Sigma^n(\fB_1   \fB_2)) \oplus K_0\Sigma^n(R)\\
\longrightarrow
\Nil_0(\Sigma^n R; \Sigma^n \fB_1 \xo_{\Sigma^n R} \Sigma^n
\fB_2) \oplus K_0\Sigma^n(R)\\
\longra
\Nil_0(\Sigma^nR;\Sigma^n\fB_1,\Sigma^n\fB_2) =
\Nil_{-n}(R;\fB_1,\fB_2).
\proved\end{multline*}
\end{proof}

\section{Applications}\label{Applications}

We indicate some applications of our main theorem (\ref{maink}).  In \fullref{First} we prove \fullref{first}(ii),
which describes the restrictions of the maps
$$\theta_!~:~K_*(R[\ol{G}]) \to
K_*(R[G])~,~\theta^!~:~K_*(R[G]) \to K_*(R[\ol{G}])$$ to the
$\wt{\Nil}$-terms, with $\theta:\,\ol{G} \to G$ the
inclusion of the canonical index 2 subgroup $\ol{G}$ for any
group $G$ over $D_\infty$. In \fullref{Second} we give the first known
example of a non-zero $\Nil$-group occurring in the $K$-theory of an
integral group ring of an amalgamated free product.  In
\fullref{Third} we sharpen the Farrell--Jones Conjecture in
$K$-theory, replacing the family of virtually cyclic groups by the
smaller family of finite-by-cyclic groups.  In \fullref{Subsec_PSL2Z}
we compute the $K_*(R[\g])$ for the modular group $\g = PSL_2(\Z)$.

\subsection{Algebraic $K$-theory over $D_\infty$}\label{First}

The overall goal here is to show that the abstract
isomorphisms $i_*$ and $j_*$ coincide with the restrictions of the
induction and transfer maps $\theta_!$ and $\theta^!$ in the group
ring setting.

\subsubsection{Twisting}
We start by recalling the algebraic $K$-theory of twisted polynomial
rings.

\begin{stt}
Consider any (unital, associative) ring $R$ and any ring
automorphism $\alpha:\,R \to R$. Let $t$ be an indeterminate over $R$
such that
\[rt~=~t\alpha(r)~~(r \in R)~.\]
For any $R$-module $P$, let $tP := \{tx ~|~ x \in P\}$ be the set
with left $R$-module structure
\[tx + t y~=~t (x+y)~,~r(tx)~=~t(\alpha(r)x) \in tP~.\]
Further endow the left $R$-module $tR$ with the $R$-bimodule
structure
\[R \times tR \times R \longrightarrow tR~;~(q,tr,s) \longmapsto t
\alpha(q)rs~.\]
The $\Nil$-category of $R$ with respect to $\alpha$ is the exact
category defined by
\[\NIL(R,\alpha)~:=~\NIL(R;tR).\]
The objects $(P,\rho)$ consist of any finitely generated  projective $R$-module $P$
and any nilpotent morphism $\rho:\,P \to tP=tRP$.  The
$\Nil$-groups are written
\[\Nil_*(R,\alpha)~:=~\Nil_*(R;tR)~,~\wt{\Nil}_*(R,\alpha)~:=~
\wt{\Nil}_*(R;tR)~,\] so that
\[\Nil_*(R,\alpha)~=~K_*(R) \oplus \wt{\Nil}_*(R,\alpha)~.\]
\end{stt}

\begin{stt}
The tensor algebra on $tR$ is the $\alpha$-twisted polynomial
extension of $R$
\[T_R(tR)~=~R_{\alpha}[t]~=~\sum\limits^{\infty}_{k=0}t^kR~.\]
Given an $R$-module $P$ there is induced an $R_{\alpha}[t]$-module
$$R_{\alpha}[t]\otimes_RP~=~P_{\alpha}[t]$$
whose elements are finite linear combinations
$\sum\limits^{\infty}_{j=0}t^jx_j$ ($x_j \in P$). Given $R$-modules
$P,Q$ and an $R$-module morphism $\rho:\,P \to tQ$, define its
extension as the $R_{\alpha}[t]$-module morphism
\[
\widehat{\rho}~=~t\rho~:~P_{\alpha}[t] \longrightarrow Q_{\alpha}[t]~;~
\sum\limits^{\infty}_{j=0}t^j x_j
\longmapsto \sum\limits^{\infty}_{j=0}t^j \rho(x_j)~.
\]
\end{stt}

\begin{stt}
Bass~\cite{Bass}, Farrell--Hsiang~\cite{FarrellHsiang}, and
Quillen~\cite{Grayson} give decompositions:
\begin{eqnarray*}
K_n(R_{\alpha}[t]) &=& K_n(R) \oplus \wt{\Nil}_{n-1}(R,\alpha)~,\\
K_n(R_{\alpha\inv}[t\inv]) &=& K_n(R) \oplus \wt{\Nil}_{n-1}(R,\alpha\inv)~,\\
K_n(R_{\alpha}[t,t\inv]) &=& K_n(1-\alpha:\,R \to R)\oplus
\wt{\Nil}_{n-1}(R,\alpha)\oplus \wt{\Nil}_{n-1}(R,\alpha\inv)~.
\end{eqnarray*}
In particular for $n = 1$, by \fullref{Thm_WaldhausenTensorRing}, there are defined split
monomorphisms:
\begin{eqnarray*}
\sigma^+_B &:& \wt{\Nil}_0(R,\alpha) \longrightarrow K_1(R_{\alpha}[t])~;~
 [P,\rho] \longmapsto \left[ P_{\alpha}[t], 1-t\rho \right]~,\\[1ex]
\sigma^-_B &:& \wt{\Nil}_0(R,\alpha\inv) \longrightarrow K_1(R_{\alpha\inv}[t\inv])~;~
[P,\rho] \longmapsto \left[ P_{\alpha\inv}[t\inv], 1-t\inv\rho \right]~,\\[1ex]
\sigma_B &=& \begin{pmatrix}\psi^+\sigma^+_B & \psi^-\sigma^-_B\end{pmatrix}~:~\wt{\Nil}_0(R,\alpha)\oplus
\wt{\Nil}_0(R,\alpha\inv) \longrightarrow K_1(R_{\alpha}[t,t\inv])~;\\
& &  ([P_1,\rho_1],[P_2,\rho_2]) \longmapsto \left[ (P_1\oplus P_2)_{\alpha}[t,t\inv], \begin{pmatrix}1-t\rho_1 & 0\\ 0 & 1-t\inv\rho_2 \end{pmatrix} \right]~.
\end{eqnarray*}
These extend to all integers $n \leqslant 1$ by the
suspension isomorphisms of \fullref{Sec_Lower}.
\end{stt}

\subsubsection{Scaling}
Next, consider the effect an inner automorphism on $\alpha$.

\begin{stt}
Suppose $\alpha,\alpha':\,R \to R$ are automorphisms satisfying
$$\alpha'(r)~=~u\alpha(r)u\inv\in R~~(r \in R)$$
for some unit $u \in R$, and that $t'$ is an indeterminate over $R$
satisfying
$$rt'~=~t'\alpha'(r)~~(r \in R)~.$$
Denote the canonical inclusions
\[
\begin{array}{ll}
\psi^+~:~R_{\alpha}[t] \longrightarrow R_{\alpha}[t,t\inv] &
\psi^-~:~R_{\alpha\inv}[t\inv] \longrightarrow R_{\alpha}[t,t\inv]\\[1ex]
\psi'^+~:~R_{\alpha'}[t'] \longrightarrow R_{\alpha'}[t',t'{}\inv] &
\psi'^-~:~R_{\alpha'{}\inv}[t'{}\inv] \longrightarrow R_{\alpha'}[t',t'{}\inv]~.
\end{array}
\]
\end{stt}

\begin{stt}\label{scaling1}
The various polynomial rings are related by \textbf{scaling
isomorphisms}
$$\begin{array}{l}
\beta^+_u~:~R_{\alpha}[t] \longrightarrow R_{\alpha'}[t']~;~
t \longmapsto t'u~,\\[1ex]
\beta^-_u~:~R_{\alpha\inv}[t\inv] \longrightarrow R_{\alpha'{}\inv}[t'{}\inv]~;~
t\inv \longmapsto u\inv t'{}\inv~,\\[1ex]
\beta_u~:~R_{\alpha}[t,t\inv] \longrightarrow R_{\alpha'}[t',t'{}\inv]~;~
t \longmapsto t'u
\end{array}$$
satisfying the equations
\begin{eqnarray*}
\beta_u \circ \psi^+ &=& \psi'^+ \circ \beta_u^+ ~:~ R_\alpha[t] \longrightarrow R_{\alpha'}[t',t'{}\inv]\\[1ex]
\beta_u \circ \psi^- &=& \psi'^- \circ \beta_u^- ~:~ R_{\alpha\inv}[t\inv] \longrightarrow R_{\alpha'}[t',t'{}\inv] ~.
\end{eqnarray*}
\end{stt}

\begin{stt}\label{scaling2}
There are corresponding scaling isomorphisms of exact categories
\begin{eqnarray*}
\beta^+_u &:& \NIL(R,\alpha) \longrightarrow \NIL(R,\alpha')~;~
(P,\rho) \longmapsto (P,t'u t\inv \rho:P\to t'P)\\[1ex]
\beta^-_u &:& \NIL(R,\alpha\inv) \longrightarrow \NIL(R,\alpha'{}\inv) ~;~
(P,\rho) \longmapsto (P,t'{}\inv u t \rho:P'\to t'{}\inv P')~,
\end{eqnarray*}
where we mean
\begin{eqnarray*}
(t'u t\inv \rho)(x) &:=& t' (uy) ~\text{with}~ \rho(x) = ty,\\
(t'{}\inv u t\rho)(x) &:=& t'{}\inv (uy) ~\text{with}~ \rho(x) = t\inv y ~.
\end{eqnarray*}
\end{stt}

\begin{stt}\label{scaling3}
For all $n \leqslant 1$, the various scaling isomorphisms are related by
equations
\begin{eqnarray*}
(\beta_u^+)_* \circ \sigma_B^+ &=& \sigma_B'^+ \circ \beta_u^+ ~:~
 \wt{\Nil}_{n-1}(R,\alpha) \longrightarrow K_n(R_{\alpha'}[t'])\\[1ex]
(\beta_u^-)_* \circ \sigma_B^- &=& \sigma_B'^- \circ \beta_u^- ~:~
 \wt{\Nil}_{n-1}(R,\alpha\inv) \longrightarrow K_n(R_{\alpha'{}\inv}[t'{}\inv])\\[1ex]
(\beta_u)_* \circ \sigma_B &=& \sigma_B' \circ \left(\begin{smallmatrix}\beta_u^+ & 0\\ 0 & \beta_u^-\end{smallmatrix}\right) ~:~
 \wt{\Nil}_{n-1}(R,\alpha)\oplus \wt{\Nil}_{n-1}(R,\alpha\inv) \longrightarrow K_n(R_{\alpha'}[t',t'{}\inv]) ~.
\end{eqnarray*}
\end{stt}

\subsubsection{Group rings}
We now adapt these isomorphisms to the case of group rings $R[G]$ of
groups $G$ over the infinite dihedral group $D_\infty$. In order to
prove \fullref{lem:induction} and \fullref{transfer}, the overall
idea is to transform information about the product $t_2 t_1$ arising
from the transposition $\fB_2 \otimes \fB_1$ into information about
the product $t_2\inv t_1\inv$ arising in the second
$\wt{\Nil}$-summand of the twisted Bass decomposition. We
continue to discuss the ingredients in a sequence of statements.

\begin{stt}
Let $F$ be a group, and let $\alpha:\,F \to F$ be an automorphism.
Recall that the injective HNN extension $F\rtimes_{\alpha} \Z$ is
the set $F \times \Z$ with group multiplication
$$(x,n)(y,m)~:=~(\alpha^m(x)y,m+n) \in F\rtimes_{\alpha}\Z~.$$
Then, for any ring $R$, writing $t=(1_F,1)$ and $(x,n)=t^nx \in
F\rtimes_{\alpha}\Z$, we have
$$R[F\rtimes_{\alpha} \Z]~=~R[F]_{\alpha}[t,t\inv]~.$$
\end{stt}

\begin{stt}\label{Statement_OuterAutomorphisms}
Consider any group $G=G_1 *_F G_2$ over $D_\infty$, where
$$F~=~G_1 \cap G_2~\subset~\ol{G}
~=~F\rtimes_{\alpha}\Z~=~F\rtimes_{\alpha'}\Z ~\subset~
G~=~G_1*_FG_2~.$$ Fix elements $t_1 \in G_1-F$, $t_2\in G_2-F$, and
define elements
$$t~:=~t_1t_2 \in \ol{G},~t'~:=~t_2t_1 \in \ol{G}~,~u~:=~(t')\inv t\inv \in F~.$$
Define the automorphisms
\begin{eqnarray*}
\alpha_1 &:& F \longra F~;~x \longmapsto (t_1)\inv x t_1~,\\
\alpha_2 &:& F \longra F~;~x \longmapsto (t_2)\inv x t_2~,\\
\alpha ~:=~ \alpha_2\circ\alpha_1 &:& F \longra F~;~x \longmapsto t\inv xt~,\\
\alpha' ~:=~ \alpha_1\circ\alpha_2 &:& F \longra F~;~x \longmapsto t'{}\inv xt'
\end{eqnarray*}
such that
\[
xt~=~t \alpha(x)~,\quad xt'=~t'\alpha'(x)~,\quad
\alpha'(x)~=~u \alpha\inv(x)u\inv~~(x \in F)~.
\]
\textbf{In particular, note $\alpha'$ and $\alpha\inv$ (not
$\alpha$) are related by inner automorphism by $u$.}
\end{stt}

\begin{stt}\label{Statement_RingMaps}
Denote the canonical inclusions
\[
\begin{array}{ll}
\psi^+~:~R_{\alpha}[t] \longrightarrow R_{\alpha}[t,t\inv] &
\psi^-~:~R_{\alpha\inv}[t\inv] \longrightarrow R_{\alpha}[t,t\inv]\\[1ex]
\psi'^+~:~R_{\alpha'}[t'] \longrightarrow R_{\alpha'}[t',t'{}\inv] &
\psi'^-~:~R_{\alpha'{}\inv}[t'{}\inv] \longrightarrow R_{\alpha'}[t',t'{}\inv]~.
\end{array}
\]
The inclusion $R[F] \to R[G]$ extends to ring monomorphisms
\[
\begin{array}{ll}
\theta~:~R[F]_{\alpha}[t,t\inv] \longrightarrow R[G] &
\theta'~:~R[F]_{\alpha'}[t',t'{}\inv] \longrightarrow R[G]
\end{array}
\]
such that
\[
\mathrm{im}(\theta)~=~\mathrm{im}(\theta')~=~ R[\ol{G}]
\subset R[G]~=~R[G_1]*_{R[F]}R[G_2]~.
\]
Furthermore, the inclusion $R[F] \to R[G]$ extends to ring
monomorphisms
\[
\begin{array}{ll}
\phi = \theta \circ \psi^+~:~R[F]_{\alpha}[t] \longra R[G] & \phi' = \theta' \circ \psi'^+~:~R[F]_{\alpha'}[t'] \longra R[G]~.
\end{array}
\]
\end{stt}

\begin{stt}\label{scalingG1}
By \fullref{scaling1}, there are defined scaling isomorphisms
of rings
\begin{eqnarray*}
\beta^+_u &:& R[F]_{\alpha\inv}[t\inv]
\longrightarrow
R[F]_{\alpha'}[t'] ~;~t\inv \longmapsto t'u~,\\[1ex]
\beta^-_u &:& R[F]_{\alpha}[t] \longrightarrow
R[F]_{\alpha'{}\inv}[t'{}\inv] ~;~
t \longmapsto u\inv t'{}\inv~,\\[1ex]
\beta_u &:& R[F]_{\alpha}[t,t\inv]
\longrightarrow  R[F]_{\alpha'}[t',t'{}\inv]~;~
t \longmapsto u\inv t'{}\inv
 \end{eqnarray*}
which satisfy the equations
\begin{eqnarray*}
\beta_u \circ \psi^- &=& \psi'^+ \circ \beta_u^+ ~:~ R[F]_{\alpha\inv}[t\inv] \longrightarrow R[F]_{\alpha'}[t',t'{}\inv]\\[1ex]
\beta_u \circ \psi^+ &=& \psi'^- \circ \beta_u^- ~:~ R[F]_{\alpha}[t] \longrightarrow R[F]_{\alpha'}[t',t'{}\inv]\\[1ex]
\theta &=& \theta' \circ \beta_u ~:~ R[F]_{\alpha}[t,t\inv] \longrightarrow R[G]~.
\end{eqnarray*}
\end{stt}

\begin{stt}\label{scalingG2}
By \fullref{scaling2}, there are scaling isomorphisms of exact
categories
\begin{eqnarray*}
\beta^+_u &:& \NIL(R[F],\alpha\inv) \longrightarrow \NIL(R[F],\alpha')~;~
(P,\rho) \longmapsto (P,t'u t\rho)~,\\[1ex]
\beta^-_u &:& \NIL(R[F],\alpha) \longrightarrow \NIL(R[F],\alpha'{}\inv)~;~
(P,\rho) \longmapsto (P,t'{}\inv u t\inv\rho)~.
\end{eqnarray*}
\end{stt}

\begin{stt}\label{scalingG3}
By \fullref{scaling3}, for all $n \leqslant 1$, the various scaling
isomorphisms are related by:
\begin{eqnarray*}
(\beta_u^+)_* \circ \sigma_B^- &=& \sigma_B'^+ \circ \beta_u^+ ~:~ \wt{\Nil}_{*-1}(R[F],\alpha\inv) \longrightarrow K_*(R[F]_{\alpha'}[t'])\\[1ex]
(\beta_u^-)_* \circ \sigma_B^+ &=& \sigma_B'^- \circ \beta_u^- ~:~ \wt{\Nil}_{*-1}(R[F],\alpha) \longrightarrow K_*(R[F]_{\alpha'{}\inv}[t'{}\inv])\\[1ex]
(\beta_u)_* \circ \sigma_B &=& \sigma_B' \circ \left(\begin{smallmatrix} 0 & \beta^+_u \\
\beta^-_u & 0 \end{smallmatrix}\right) ~:~ \wt{\Nil}_{*-1}(R[F],\alpha)\oplus \wt{\Nil}_{*-1}(R[F],\alpha\inv) \longrightarrow K_*(R[F]_{\alpha'}[t',t'{}\inv]) ~.
\end{eqnarray*}
\end{stt}

\subsubsection{Transposition}
Next, we study the effect of transposition of the bimodules $\fB_1$
and $\fB_2$ in order to relate $\alpha$ and $\alpha'$. In
particular, there is no mention of $\alpha\inv$ in this section.

\begin{stt}
The $R[F]$-bimodules
\[
\fB_1~=~R[G_1-F]~=~t_1R[F]~,~\fB_2~=~R[G_2-F]~=~t_2R[F]
\]
are free left and right $R[F]$-modules of rank one. The
$R[F]$-bimodule isomorphisms
\begin{align*}
&\fB_1\otimes_{R[F]}\fB_2 \longra tR[F]~;~t_1x_1\otimes t_2x_2 \longmapsto
t \alpha_2(x_1)x_2\\
&\fB_2\otimes_{R[F]}\fB_1 \longra t'R[F]~;~t_2x_2\otimes t_1x_1 \longmapsto
t' \alpha_1(x_2)x_1
\end{align*}
shall be used to make the identifications
$$\begin{array}{ll}
\fB_1\otimes_{R[F]}\fB_2~=~tR[F]~,&\NIL(R[F];\fB_1\otimes_{R[F]}\fB_2)~=~\NIL(R[F],\alpha)~,\\[1ex]
\fB_2\otimes_{R[F]}\fB_1~=~t'R[F]~,&\NIL(R[F];\fB_2\otimes_{R[F]}\fB_1)~=~\NIL(R[F],\alpha')~.
\end{array}$$
\end{stt}

\begin{stt}
\fullref{maink} gives inverse isomorphisms
\begin{eqnarray*}
i_* &:& \wt{\Nil}_*(R[F],\alpha) \longrightarrow  \wt{\Nil}_*(R[F];\fB_1,\fB_2)~,\\
j_* &:& \wt{\Nil}_*(R[F];\fB_1,\fB_2) \longrightarrow
\wt{\Nil}_*(R[F],\alpha)
\end{eqnarray*}
which for $*=0$ are given by
\begin{eqnarray*}
i_* &:& \wt{\Nil}_0(R[F],\alpha) \longrightarrow \wt{\Nil}_0(R[F];\fB_1,\fB_2)~;~
[P,\rho] \longmapsto [P, t_2P,\rho,1]~,\\
j_* &:& \wt{\Nil}_0(R[F];\fB_1,\fB_2)
\longrightarrow \wt{\Nil}_0(R[F],\alpha)~;~
[P_1,P_2,\rho_1,\rho_2]\longmapsto [P_1,\rho_2\circ \rho_1]~.
\end{eqnarray*}
\end{stt}

\begin{stt}
Similarly, there are defined inverse isomorphisms
\begin{eqnarray*}
i'_* &:& \wt{\Nil}_*(R[F],\alpha') \longrightarrow \wt{\Nil}_*(R[F];\fB_2,\fB_1)~,\\
j'_* &:& \wt{\Nil}_*(R[F];\fB_2,\fB_1)
\longrightarrow \wt{\Nil}_*(R[F],\alpha')
\end{eqnarray*}
which for $*=0$ are given by
\begin{eqnarray*}
i'_* &:& \wt{\Nil}_0(R[F],\alpha') \longrightarrow \wt{\Nil}_0(R[F];\fB_2,\fB_1)~;~
[P',\rho'] \longmapsto [P',t_1P',\rho',1]~,\\
j'_* &:& \wt{\Nil}_0(R[F];\fB_2,\fB_1) \longrightarrow
\wt{\Nil}_0(R[F],\alpha')~;~
[P_2,P_1,\rho_2,\rho_1]\longmapsto [P_2,\rho_1\circ \rho_2]~.
\end{eqnarray*}
\end{stt}

\begin{stt}\label{Statement_Transposition}
The transposition isomorphism of exact categories
\[
\tau_A ~:~\NIL(R[F];\fB_1,\fB_2) \longrightarrow \NIL(R[F];\fB_2,\fB_1)~;~
(P_1,P_2,\rho_1,\rho_2) \longmapsto (P_2,P_1,\rho_2,\rho_1)
\]
induces isomorphisms
\begin{eqnarray*}
\tau_A &:& \Nil_*(R[F];\fB_1,\fB_2)~\cong~\Nil_*(R[F];\fB_2,\fB_1)~,\\[1ex]
\tau_A &:& \wt{\Nil}_*(R[F];\fB_1,\fB_2)~\cong~\wt{\Nil}_*(R[F];\fB_2,\fB_1)~.
\end{eqnarray*}
Note, by \fullref{maink}, the composites
\begin{eqnarray*}
\tau_B &:=& j'_* \circ \tau_A \circ i_*~:~\wt{\Nil}_*(R[F],\alpha) \longrightarrow
\wt{\Nil}_*(R[F],\alpha')~,\\[1ex]
\tau'_B &:=& j_* \circ \tau\inv_A \circ i'_*~:~\wt{\Nil}_*(R[F],\alpha') \longrightarrow
\wt{\Nil}_*(R[F],\alpha)
\end{eqnarray*}
are inverse isomorphisms, which for $*=0$ are given by
\begin{eqnarray*}
\tau_B &:& \wt{\Nil}_0(R[F],\alpha) \longrightarrow
\wt{\Nil}_0(R[F],\alpha')~;~[P,\rho] \longmapsto
[t_2P,t_2\rho]~,\\[1ex]
\tau'_B &:& \wt{\Nil}_0(R[F],\alpha') \longrightarrow
\wt{\Nil}_0(R[F],\alpha)~;~[P',\rho'] \longmapsto
[t_1P',t_1\rho']~.
\end{eqnarray*}
Furthermore, note that the various transpositions are related by the
equation
\[
\tau_A \circ i_* ~=~ i'_* \circ \tau_B ~:~ \wt{\Nil}_*(R[F],\alpha) \longrightarrow \wt{\Nil}_*(R[F];\fB_2,\fB_1)~.
\]
\end{stt}

\begin{stt}\label{Statement_SigmaA}
Recall from \fullref{Waldhausen} that there is a split
monomorphism
\[
\sigma_A~:~\wt{\Nil}_{n-1}(R[F];\fB_1,\fB_2) \longra K_n(R[G])
\]
such that the $n=1$ case is given by
\begin{multline*}
\sigma_A~:~\wt{\Nil}_0(R[F];\fB_1,\fB_2) \longra K_1(R[G])~;\\
[P_1,P_2,\rho_1,\rho_2] \longmapsto \left[ P_1[G] \oplus P_2[G],
\begin{pmatrix} 1 & t_2\rho_2 \\ t_1\rho_1 & 1\end{pmatrix}
\right]~.
\end{multline*}
Elementary row and column operations produce an equivalent
representative:
\[
\begin{pmatrix}
1 & - t_2 \rho_2\\
0 & 1
\end{pmatrix}
\begin{pmatrix}
1 & t_2 \rho_2\\
t_1 \rho_1 & 1
\end{pmatrix}
\begin{pmatrix}
1 & 0\\
-\rho_1 & 1
\end{pmatrix}
= \begin{pmatrix}
1 - t \rho_2 \rho_1 & 0\\
0 & 1
\end{pmatrix}~.
\]
Thus the $n=1$ case satisfies the equations (similarly
for the second equality):
\[
\sigma_A[P_1,P_2,\rho_1,\rho_2] ~=~ \left[ P_1[G], 1-t\rho_2\rho_1 \right]
~=~ \left[ P_2[G], 1-t'\rho_1\rho_2 \right]~.
\]
Therefore for all $n \leqslant 1$, the split monomorphism $\sigma_A'$,
associated to the amalgamated free product $G=G_2*_FG_1$, satisfies
the equation
\[
\sigma_A~=~\sigma_A' \circ \tau_A~:~
\wt{\Nil}_{n-1}(R[F];\fB_1,\fB_2) \longra K_n(R[G]) ~.
\]
\end{stt}

\subsubsection{Induction}
We analyze the effect of induction maps on $\wt{\Nil}$-summands.

\begin{stt}\label{Statement_meat}
Recall from \fullref{maink} the isomorphism
\begin{multline*}
i_*~:~\wt{\Nil}_{*-1}(R[F],\alpha)~=~
\wt{\Nil}_{*-1}(R[F];\fB_1\otimes_{R[F]}\fB_2)
\longra \wt{\Nil}_{*-1}(R[F];\fB_1,\fB_2)~;\\
[P,\rho] \longmapsto [P,t_2P,\rho,1]~.
\end{multline*}
Let $(P,\rho)$ be an object in the exact category
$\NIL(R[F],\alpha)$. By \fullref{Statement_SigmaA}, note
\[
\sigma_A i_* [P,\rho] ~=~ \sigma_A [P,t_2 P, \rho,1] ~=~ [P[G],1-t\rho] ~=~ \phi_!\sigma_B^+ [P,\rho]~.
\]
Thus, for all $n \leqslant 1$, we obtain the key equality
\[
\sigma_A \circ i_* ~=~ \phi_! \circ \sigma_B^+ ~:~ \wt{\Nil}_{n-1}(R[F],\alpha) \longrightarrow K_n(R[G]) ~.
\]
\end{stt}


\begin{lem}\label{lem:induction}
Let $n \leqslant 1$ be an integer.
The split monomorphisms $\sigma_A,\sigma'_A,\sigma_B^+,\sigma_B'^+$ are related by a commutative diagram
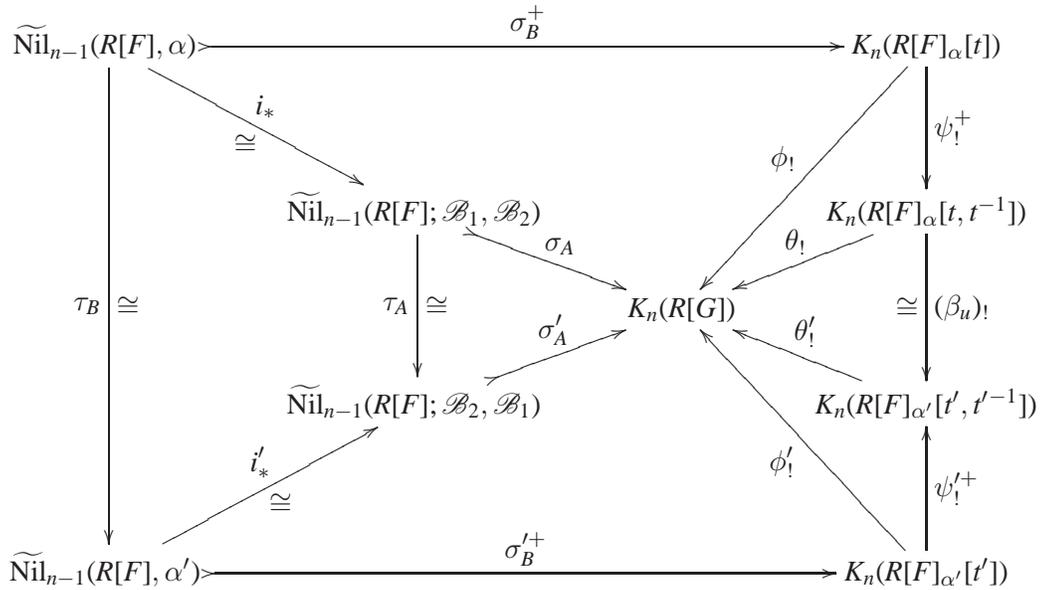
\[\xymatrix@R-5pt{
\wt{\Nil}_{n-1}(R[F],\alpha)
~\ar@{>->}[rrr]^-{\di{\sigma^+_B}} \ar[dddddd]_-{\di{\tau_B}}^-{\di{\cong}}
\ar[ddr]^-{\di{i_*}}_-{\di{\cong}} &&&
K_n(R[F]_{\alpha}[t]) \ar[dddl]_-{\di{\phi_!}} \ar[dd]^-{\di{\psi_!^+}} \\
&&&\\
&\wt{\Nil}_{n-1}(R[F];\fB_1,\fB_2)
~\ar@{>->}[dr]^-{\di{\sigma_A}}
\ar[dd]^-{\di{\cong}}_-{\di{\tau_A}}
&& K_n(R[F]_{\alpha}[t,t\inv])\ar[dl]_-{\di{\theta_!}}
\ar[dd]_-{\di{\cong}}^-{\di{(\beta_u)_!}}\\
& & K_n(R[G])  & \\
&\wt{\Nil}_{n-1}(R[F];\fB_2,\fB_1)
~\ar@{>->}[ur]^-{\di{\sigma'_A}}  &&K_n(R[F]_{\alpha'}[t',t'{}\inv])\ar[ul]_-{\di{\theta'_!}}
\\
&&&\\
\wt{\Nil}_{n-1}(R[F],\alpha')
~\ar@{>->}[rrr]^-{\di{\sigma_B'^+}} \ar[uur]^-{\di{i'_*}}_-{\di{\cong}} &&&
K_n(R[F]_{\alpha'}[t'])\ar[uuul]^-{\di{\phi_!'}} \ar[uu]_-{\di{\psi_!'^+}}}
\]
\end{lem}

\begin{proof}
Commutativity of the various parts follow from the following implications:
\begin{itemize}
\item
\fullref{Statement_RingMaps} gives $\phi_! = \theta_!
\circ \psi_!^+$ and $\phi_!' = \theta_!' \circ \psi_!'^+$

\item
\fullref{scalingG1} gives $\theta_! = \theta_!' \circ
(\beta_u)_!$

\item
\fullref{Statement_Transposition} gives $\tau_A \circ
i_* = i_*' \circ \tau_B$

\item
\fullref{Statement_SigmaA} gives $\sigma_A = \sigma_A'
\circ \tau_A$

\item
\fullref{Statement_meat} gives $\sigma_A \circ i_* =
\phi_! \circ \sigma_B^+$ and $\sigma_A' \circ i_*' = \phi_!'
\circ \sigma_B'^+$.\qedhere
\end{itemize}
\end{proof}

Observe the action of $G/\ol{G}$ on $K_n(R[G])$ is inner, hence is trivial.  However, the action of $C_2 = G/\ol{G}$ on $K_n(R[\ol{G}])$ is outer, induced by, say $c_1:\, \ol{G} \to \ol{G} ~;~ y \mapsto t_1 y (t_1)\inv$. (Note $c_1$ may not have order two.)
This $C_2$-action on $K_n(R[\ol{G}])$ is non-trivial, as follows.

\begin{prop}\label{prop:induction}
Let $n \leq 1$ be an integer.
The induced map $\theta_!$ is such that there is a commutative diagram
\[\begin{diagram}
\node{\wt{\Nil}_{n-1}(R[F],\alpha)\oplus \wt{\Nil}_{n-1}(R[F],\alpha\inv)}
\arrow{s,b}{\begin{pmatrix} i_* & \tau_A\inv i'_* \beta_u^+ \end{pmatrix}}
\arrow{e,t}{\di{\sigma_B}}
\node{K_n(R[\ol{G}])}
\arrow{s,b}{\di{\theta_!}}\\
\node{\wt{\Nil}_{n-1}(R[F];\fB_1,\fB_2)}
\arrow{e,t}{\di{\sigma_A}}
\node{K_n(R[G])}
\end{diagram}\]
Furthermore, there is a $C_2$-action on the upper left hand corner which interchanges the two $\Nil$-summands, and all maps are  $C_2$-equivariant. Here, the action of $C_2 = G/\ol{G}$ on the upper right is given by $(c_1)_!$, and the $C_2$-action on each lower corner is trivial.
\end{prop}

\begin{proof}
First, we check commutativity of the square on each $\Nil$-summand:
\begin{itemize}
\item
\fullref{lem:induction} gives $\sigma_A \circ i_* = \phi_! \circ \sigma_B^+ = \theta_! \circ
\psi_!^+ \circ \sigma_B^+ = \theta_! \circ
\sigma_B|\wt{\Nil}_{n-1}(R[F],\alpha)$

\item
\fullref{scalingG3} and \fullref{scalingG1} give
$\sigma_A \circ \tau_A\inv \circ i_*' \circ \beta_u^+ =
\sigma_A' \circ i_*' \circ \beta_u^+ = \phi_!' \circ \sigma_B'^+
\circ \beta_u^+ = \theta_! \circ (\beta_u)_!\inv \circ
\psi_!'^+ \circ (\beta_u^+)_! \circ \sigma_B^- = \theta_! \circ
\psi_!^- \circ \sigma_B^- = \theta_! \circ \sigma_B |
\wt{\Nil}_{n-1}(R[F],\alpha\inv)$.
\end{itemize}

Next, define the involution $\SmMatrix{0 & \eps_*\\ \eps_*\inv & 0}$ on $\wt{\Nil}_{n-1}(R[F],\alpha) \oplus \wt{\Nil}_{n-1}(R[F],\alpha\inv)$ by
\[
\eps := (\alpha_1\inv)_! \circ \beta_u^+ ~:~ \NIL(R[F],\alpha\inv) \longrightarrow \NIL(R[F],\alpha).
\]
Here, the automorphism $\alpha_1\inv: F \to F$ was defined in \ref{Statement_OuterAutomorphisms} by $x \mapsto t_1 x (t_1)\inv$ and is the restriction of $c_1$.
It remains to show $\sigma_B$ and $\SmMatrix{i_* & \tau_{A*}\inv i'_* \beta_u^+}$ are $C_2$-equivariant, that is:
\begin{eqnarray}
\label{eqn:C2right} (c_1)_! \circ \psi^-\sigma_B^- &=& \psi^+\sigma_B^+ \circ \eps\\
\label{eqn:C2down} \tau_{A*}\inv i'_* \beta_u^+ &=& i_* \circ \eps.
\end{eqnarray}

Observe that the induced ring automorphism $(c_1)_!: R[F]_\alpha[t,t\inv] \to R[F]_\alpha[t,t\inv]$ restricts to a ring isomorphism
\[
(c_1)^+_! ~:~ R[F]_{\alpha\inv}[t\inv] \longra R[F]_\alpha[t] ~;~ x \longmapsto \alpha_1\inv(x) ~;~ t\inv \longmapsto t \, \alpha_1\inv(u).
\]
Then $(c_1)_! \circ \psi^- = \psi^+ \circ (c_1)^+_!$. So (\ref{eqn:C2right}) follows from the commutative square
\[
(c_1)^+_! \circ \sigma_B^- ~=~ \sigma_B^+ \circ (\alpha_1\inv)_! \beta_u^+ ~:~ \wt{\Nil}_{n-1}(R[F],\alpha\inv) \longra K_n(R[F]_\alpha[t]),
\]
which can be verified by formulas for $n=1$ and extends to $n < 1$ by variation of $R$.

Observe that (\ref{eqn:C2down}) follows from the existence of an exact natural transformation
\[
T ~:~ \tau_A\inv \circ i' \to i \circ (\alpha_1\inv)_! ~:~ \NIL(R[F],\alpha') \longra \NIL(R[F]; t_1 R[F], t_2 R[F])
\]
defined on objects $(P, \rho: P \to t'P = t_2 t_1 P)$ by the rule
\[
T_{(P,\rho)} ~:=~ (1,\rho) ~:~ (t_1 P, P, 1, \rho) \longra (t_1 P, t'P, t_1\rho, 1);
\]
a key observation from \ref{Statement_OuterAutomorphisms} is the isomorphism $R[F] \xo_{c_1} P \to t_1 P ~;~ x \xo p \mapsto \alpha_1(x) p$.
\end{proof}

\subsubsection{Transfer}
We analyze the effect of transfer maps on
$\wt{\Nil}$-summands.

\begin{stt}
Given an $R[G]$-module $M$, let $M^!$ be the abelian group $M$ with
$R[\ol{G}]$-action the restriction of the $R[G]$-action.  The transfer functor
of exact categories
$$\theta^!~:~\PROJ(R[G]) \longra \PROJ(R[\ol{G}])~;~M \longmapsto M^!$$
induces the transfer maps in algebraic $K$-theory
$$\theta^!~:~K_*(R[G]) \longra K_*(R[\ol{G}])~.$$
The exact functors of \fullref{maink} combine to give an exact
functor
$$\begin{array}{l}
\begin{pmatrix}j\\ j'\end{pmatrix} ~:~\NIL(R[F];\fB_1,\fB_2) \longra \NIL(R[F],\alpha)
\times \NIL(R[F],\alpha')~;\\[1ex]
\hskip150pt
[P_1,P_2,\rho_1,\rho_2]  \longmapsto \big([P_1,\rho_2\circ \rho_1],[P_2,\rho_1\circ \rho_2]\big)
\end{array}$$
inducing a map between reduced $\Nil$-groups
\[
\begin{pmatrix}j_* \\ j'_*\end{pmatrix}~:~\wt{\Nil}_*(R[F];\fB_1,\fB_2) \longra \wt{\Nil}_*(R[F],\alpha)
\oplus \wt{\Nil}_*(R[F],\alpha')~.
\]
\end{stt}

\begin{prop}\label{transfer}
Let $n \leqslant 1$ be an integer. The transfer map $\theta^!$ restricts
to the isomorphism $j_*$ in a commutative diagram
\[
\xymatrix@C+10pt@R+10pt{
\wt{\Nil}_{n-1}(R[F];\fB_1,\fB_2) \ar[d]_-{\begin{pmatrix}
j_* \\ (\beta_u^+)\inv j'_* \end{pmatrix}}
~~\ar@{>->}[rr]^-{\di{\sigma_A}} & & K_n(R[G]) \ar[d]^-{\di{\theta^!}} \\
\wt{\Nil}_{n-1}(R[F],\alpha)\oplus
\wt{\Nil}_{n-1}(R[F],\alpha\inv)
~~\ar@{>->}[rr]^-{\begin{pmatrix}\psi^+\sigma_B^+ & \beta_u \psi^- \sigma_B^-\end{pmatrix}} & & K_n(R[\ol{G}])~.}
\]
\end{prop}

\begin{proof}
Using the suspension isomorphisms of \fullref{Sec_Lower}, we may
assume $n=1$. Let $(P_1,P_2,\rho_1,\rho_2)$ be an object in
$\NIL(R[F];\fB_1,\fB_2)$. Define an $R[G]$-module automorphism
\[
f~:=~\begin{pmatrix} 1 & t_2 \rho_2 \\ t_1 \rho_1 & 1
\end{pmatrix}~: P_1[G] \oplus P_2[G] \longra P_1[G] \oplus P_2[G] ~.
\]
By \fullref{Waldhausen}, we have $[f] =
\sigma_A[P_1,P_2,\rho_1,\rho_2] \in K_1(R[G])$. Note the transfer is
\[
\theta^!(f)~=~\begin{pmatrix} 1 & t_2 \rho_2 & 0  & 0 \\
t_1 \rho_1 & 1 & 0 & 0 \\
0 & 0 & 1 & t_1 \rho_1 \\
0 & 0 & t_2 \rho_2 & 1 \end{pmatrix}
\]
as an $R[\ol{G}]$-module automorphism of $P_1[\ol{G}]
\oplus t_1P_2[\ol{G}] \oplus P_2[\ol{G}] \oplus
t_1P_1[\ol{G}]$. Furthermore, elementary row and column
operations produce a diagonal representation:
\[
\begin{pmatrix} 1 & -t_2 \rho_2 & 0  & 0 \\
0 & 1 & 0 & 0 \\
0 & 0 & 1 & -t_1 \rho_1 \\
0 & 0 & 0 & 1 \end{pmatrix} \theta^!(f) \begin{pmatrix} 1 & 0 & 0  & 0 \\
-t_1 \rho_1 & 1 & 0 & 0 \\
0 & 0 & 1 & 0 \\
0 & 0 & -t_2 \rho_2 & 1 \end{pmatrix}~=~
\begin{pmatrix} 1-t'\rho_2\rho_1 & 0 & 0  & 0 \\
0 & 1 & 0 & 0 \\
0 & 0 & 1-t\rho_1\rho_2 & 0 \\
0 & 0 & 0 & 1 \end{pmatrix}~.
\]
So $\theta^![f] = [1-t'\rho_2\rho_1] + [1 - t\rho_1\rho_2]$. Thus we
obtain a commutative diagram
$$\begin{array}{l}
\xymatrix@C+10pt@R+10pt{
\wt{\Nil}_{0}(R[F];\fB_1,\fB_2)
\ar[d]_-{\begin{pmatrix} j_* \\ j'_* \end{pmatrix}}
~~\ar@{>->}[r]^-{\di{\sigma_A}} & K_1(R[G]) \ar[d]^-{\di{\theta^!}} \\
\wt{\Nil}_{0}(R[F],\alpha)\oplus
\wt{\Nil}_{0}(R[F],\alpha')
~~\ar@{>->}[r]^-{\di{\begin{pmatrix}\psi^+ \sigma^+_B & \psi'^+ \sigma_B'^+\end{pmatrix}}} & K_1(R[\ol{G}])}
\end{array}$$
Finally, by \fullref{scalingG3} and \fullref{scalingG1},
note
\[
\psi'^+ \circ \sigma_B'^+ \circ \beta_u^+ = \psi'^+ \circ \beta_u^+ \circ \sigma_B^-
 = \beta_u \circ \psi^- \circ \sigma_B^- ~.
\proved\]
\end{proof}

\subsection{Waldhausen Nil}\label{Second}

Examples of bimodules originate from group rings of
amalgamated product of groups.

\begin{dfn}\label{Defn_almostnormal}
A subgroup $H$ of a group $G$ is \textbf{almost-normal} if $|H : H
\cap x H x\inv| < \infty$ for every $x \in G$. In other words, $H$ is commensurate with all its conjugates. Equivalently, $H$ is
an almost-normal subgroup of $G$ if every $(H,H)$-double coset $HxH$
is both a union of finitely many left cosets $gH$ and a
union of finitely many right cosets $Hg$.
\end{dfn}

\begin{rem} Almost-normal subgroups arise in the Shimura theory of
automorphic functions, with $(G,H)$ called a Hecke pair.  Here are two
sufficient conditions for a subgroup $H \subset G$ to be almost-normal:
if $H$ is a finite-index subgroup of $G$, or if $H$ is a normal
subgroup of $G$.  Examples of almost-normal subgroups are
given in \cite[page~9]{Krieg}.
\end{rem}

Here is our reduction for a certain class of group rings,
specializing the General Algebraic Semi-splitting of \fullref{maink}.

\begin{cor}\label{Cor_GroupNil}
Let $R$ be a ring. Let $G = G_1 *_F G_2$ be an injective amalgamated
product of groups over a subgroup $F$ of $G_1$ and $G_2$.
Suppose $F$ is an almost-normal subgroup of $G_2$. Then, for all $n
\in \Z$, there is an isomorphism of abelian groups:
\[
j_*:\, \wt {\Nil}_n(R[F];R[G_1-F], R[G_2-F]) \longra \wt \Nil_{n}(R[F]; R[G_1-F]
\otimes_{R[F]} R[G_2-F]).
\]
\end{cor}

\begin{proof}
Consider the set $J := (F \backslash G_2 \slash F) - F$ of
non-trivial double cosets. Let $\mathcal{I}$ be the poset of all
finite subsets of $J$, partially ordered by inclusion. Note, as
$R[F]$-bimodules:
\[
R[G_2-F] = \colim_{I \in \mathcal{I}} R[I] \quad\text{where}\quad
R[I] := \bigoplus_{FgF \in I} R[FgF].
\]
Since $F$ is an almost-normal subgroup of $G_2$, each
$R[F]$-bimodule $R[I]$ is a finitely generated  free (hence
projective) left $R[F]$-module. Observe that $\mathcal{I}$ is a filtered poset: if
$I, I' \in \mathcal{I}$ then $I \cup I' \in \mathcal{I}$. Therefore
we are done by \fullref{maink}.
\end{proof}

The case of $G=D_\infty = \Z_2 * \Z_2$ has a particularly simple
form.

\begin{cor}\label{dcor}
Let $R$ be a ring and $n \in \Z$.  There are natural isomorphisms:
\begin{enumerate}
\item $ \wt {\Nil}_n(R;R,R) \cong \wt
    \Nil_{n}(R) $
\item $K_n(R[D_\infty]) \cong (K_n(R[\Z_2]) \oplus K_n(R[\Z_2]))/K_n(R) ~\oplus~ \wt{\Nil}_{n-1}(R)
    $.
\end{enumerate}
\end{cor}

\begin{proof}
Part (i) follows from \fullref{Cor_GroupNil} with $F=1$ and
$G_i=\Z_2$. Then Part (ii) follows from Waldhausen's exact sequence
(\ref{Waldhausen}), where the group retraction $\Z_2 \to 1$
induces a splitting of the map $K_n(R) \to K_n(R[\Z_2]) \x
K_n(R[\Z_2])$.
\end{proof}

\begin{exm}\label{Exm_NonzeroWaldNil}
Consider the group $G = G_0 \times D_{\infty}$ where $G_0 = \Z_2
\times \Z_2 \times \Z$.  Since $G$ surjects onto the infinite
dihedral group, there is an amalgamated product
decomposition
\[
G ~=~ (G_0 \x \Z_2) *_{G_0} (G_0 \x \Z_2)
\]
with the corresponding index 2 subgroup
\[
\ol G ~=~ G_0 \x \Z.
\]
\fullref{dcor}(1) gives an isomorphism
\[
\wt{\Nil}_{-1}(\Z[G_0]; \Z[G_0], \Z[G_0]) ~\cong~ \wt{\Nil}_{-1}(\Z[G_0]).
\]
On the other hand, Bass showed that the latter group is an
infinitely generated abelian group of exponent a power of two
\cite[XII, 10.6]{Bass}. Hence, by Waldhausen's algebraic
$K$-theory decomposition result, $\Wh(G)$ is infinitely generated
due to $\Nil$ elements. Now construct a codimension 1, finite
CW-pair $(X,Y)$ with $\pi_1 X = G$ realizing the above amalgamated
product decomposition -- for example, let $Y \to Z$ be a map of
connected CW-complexes inducing the first factor inclusion $G_0
\to G_0 \times \Z_2$ on the fundamental group and let $X$ be the
double mapping cylinder of $Z \leftarrow Y \rightarrow Z$. Next
construct a homotopy equivalence $f :\, M \to X$ of finite
CW-complexes whose torsion $\tau(f) \in \Wh(G)$ is a non-zero
$\Nil$ element. Then $f$ is non-splittable along $Y$ by
Waldhausen~\cite{Waldhausen_1969} (see \fullref{W}). \textbf{This is
the first explicit example of a non-zero Waldhausen $\wt{\Nil}$ group
and a non-splittable homotopy equivalence in the two-sided case.}
\end{exm}

\subsection{Farrell--Jones Conjecture}\label{Third}

The Farrell--Jones Conjecture asserts the family of virtually
cyclic subgroups is a ``generating'' family for $K_n(R[G])$.  In
this section we apply our main theorem to show the Farrell--Jones Conjecture holds up to dimension one if
and only if the smaller family of finite-by-cyclic subgroups is a
generating family for $K_n(R[G])$ up to dimension one.

Let $\Or G$ be the orbit category of a group $G$; objects are
$G$-sets $G/H$ where $H$ is a subgroup of $G$ and morphisms are
$G$-maps. Davis--L\"uck~\cite{DL} defined a functor $\bK_R :\, \Or G
\to \Spectra$ with the key property $\pi_n\bK_R(G/H) = K_n(R[H])$.
The utility of such a functor is that it allows the definition of
an equivariant homology theory, indeed for a $G$-CW-complex $X$,
one defines
\[
H^G_n(X; \bK_R) ~:=~ \pi_n(\map_G(-,X)_+ \wedge_{\Or G} \bK_R(-))
\]
(see \cite[Sections 4,~7]{DL} for basic properties).  Note that
the ``coefficients'' of the homology theory are given by
$H^G_n(G/H; \bK_R) = K_n(R[H])$.

A \textbf{family} $\cF$ of subgroups of $G$ is a nonempty set of
subgroups closed under conjugation and taking subgroups.  For such
a family, $E_{\cF}G$ is the classifying space for $G$-actions with
isotropy in $\cF$.  It is characterized up to $G$-homotopy type as
a $G$-CW-complex so that $(E_{\cF}G)^H$ is contractible for
subgroups $H \in \cF$ and is empty for subgroups $H \not \in \cF$.
Four relevant families are $\fin \subset \fbc \subset \vc \subset
\all$, the families of finite subgroups, finite-by-cyclic,
virtually cyclic subgroups and all subgroups respectively.  Here
\begin{eqnarray*}
\fbc &:=& \fin \cup \{H < G \mid H \cong F \rtimes \Z \text{ with $F$  finite}\}\\
\vc &:=& \{H < G \mid \exists \text{ cyclic } C < H \text{ with finite index}\}.
\end{eqnarray*}
The Farrell--Jones Conjecture in $K$-theory for the group $G$
\cite{FJiso,DL} asserts an isomorphism:
\[
H^G_n(E_{\vc}G;\bK_R) \longra H^G_n(E_{\all}G;\bK_R) = K_n(R[G]).
\]

We now state a more general version, the Fibered Farrell--Jones
Conjecture.  Let $\varphi:\, \g \to G$ be a group homomorphism. If
$\cF$ is a family of subgroups of $G$, define the family of
subgroups
\[
\varphi^*\cF ~:=~ \{ H < \g \mid \varphi(H) \in \cF\} .
\]
The Fibered Farrell--Jones Conjecture in $K$-theory for the group
$G$ asserts, for every ring $R$ and homomorphism $\varphi: \g \to G$, that following induced map is an isomorphism:
\[
H^\g_n(E_{\varphi^*\vc(G)}\g;\bK_R) \longra H^\g_n(E_{\varphi^*\all(G)}\g;\bK_R) = K_n(R[\g]).
\]

The following theorem was proved for all $n$ in \cite{DQR} using controlled topology.  We give a proof below up to dimension one using only algebraic topology.

\begin{thm} \label{isom_conj}
Let $\varphi:\, \g \to G$ be an homomorphism of groups.  Let $R$ be any
ring. The inclusion-induced map
\[
H^\g_n(E_{\varphi^*\fbc(G)}\g;\bK_R) \longra H^\g_n(E_{\varphi^*\vc(G)}\g;\bK_R).
\]
is an isomorphism for all integers $n < 1$ and an epimorphism for $n =1$.
\end{thm}

Hence we propose a sharpening of the Farrell--Jones Conjecture in algebraic $K$-theory.

\begin{conj}
Let $G$ be a discrete group, and let $R$ be a ring. Let $n$ be an integer.

\begin{enumerate}
\item There is an isomorphism:
\[
H_n^G(E_{\fbc}G; \bK_R) \longra H_n^G(E_{\all}G; \bK_R) = K_n(R[G]).
\]
\item For any homomorphism $\varphi:\, \g \to G$ of groups, there
    is an isomorphism:
\[
H^\g_n(E_{\varphi^*\fbc(G)}\g;\bK_R) \longra H^\g_n(E_{\all}\g;\bK_R) = K_n(R[\g]).
\]
\end{enumerate}
\end{conj}


The proof of \fullref{isom_conj} will require three
auxiliary results, some of which we quote from other sources.  The
first is a variant of Theorem A.10 of Farrell--Jones~\cite{FJiso}, whose proof is identical to the proof of Theorem A.10.

\begin{tp}\label{TransitivityPrinciple}
Let $\cF \subset \cG$ be families of subgroups of a group $\g$.
Let $\bE :\, \Or  \g \to \Spectra$ be a functor.   Let $N \in \Z \cup \{\infty\}$.   If for all $H \in \cG - \cF$, the assembly map
\[
H_n^H(E_{\cF | H}H; \bE) \longra H_n^H(E_\all H; \bE)
\]
is an isomorphism for $n < N$ and an epimorphism if $n = N$, then the map
\[
H_n^\g(E_\cF \g;\bE) \longra H_n^\g(E_\all H;\bE)
\]
is an isomorphism for $n < N$ and an epimorphism if $n = N$.
\end{tp}

Of course, we apply this principle to the families $\fbc \subset \vc$. The second auxiliary result is a well-known lemma (see \cite[Theorem~5.12]{ScottWall}), but we offer an alternative proof.

\begin{lem}\label{lem_vctypes}
Let $G$ be a virtually cyclic group.  Then either
\begin{enumerate}

\item $G$ is finite.

\item $G$ maps onto $\Z$; hence $G = F
    \rtimes_{\alpha} \Z$ with $F$ finite.

\item $G$ maps onto $D_{\infty}$; hence $G =
    G_1 *_F G_2$ with $| G_i : F| = 2$ and $F$ finite.

\end{enumerate}
\end{lem}

\begin{proof}
Assume $G$ is an infinite virtually cyclic group.
The intersection of the conjugates of a finite index, infinite cyclic subgroup is a normal, finite index, infinite cyclic subgroup $C$.
Let $Q$ be the finite quotient group.  Embed $C$ as a subgroup of index $|Q|$  in an infinite cyclic group $C'$.
There exists a unique $\Z[Q]$-module structure on $C'$ such that $C$ is a $\Z[Q]$-submodule.
Observe that the image of the obstruction cocycle under the map $H^2(Q;C) \to H^2(Q;C')$ is trivial.
Hence $G$ embeds as a finite index subgroup of a semidirect product $G' = C' \rtimes Q$.
Note $G'$ maps epimorphically to $\Z$ (if $Q$ acts trivially) or to $D_{\infty}$ (if $Q$ acts non-trivially).
In either case, $G$ maps epimorphically to a subgroup of finite index in $D_{\infty}$, which must be either infinite cyclic or infinite dihedral.
\end{proof}

In order to see how the reduced $\Nil$-groups relate
to equivariant homology (and hence to the Farrell--Jones
Conjecture), we need \cite[Lemma~3.1]{DQR}, the third auxiliary result.

\begin{lem}[Davis--Quinn--Reich]\label{DW}
Let $G$ be a group of the form $G_1 *_F G_2$ with $|G_i : F | = 2$, and let $\fac$ be the smallest family
of subgroups of $G$ containing $G_1$ and $G_2$. Let $\ol{G}$ be a group of the form $F \rtimes_{\alpha} \Z$,
and let $\ol{\fac}$ be the smallest family of subgroups of
$\ol{G}$ containing $F$. Note that $F$ need not be finite.

\begin{enumerate}

\item  The following exact sequences are split, and hence short exact:
\begin{gather*}
H^{G}_n(E_{\fac}G;\bK_R) \xrightarrow{~f_A~} H^{G}_n(E_{\all}G;\bK_R) \xrightarrow{~\eta_A~} H^{G}_n(E_{\all}G,E_{\fac}G;\bK_R)\\
H^{\ol{G}}_n(E_{\ol{\fac}}\ol{G};\bK_R) \xrightarrow{~f_B~} H^{\ol{G}}_n(E_{\all}\ol{G};\bK_R) \xrightarrow{~\eta_B~} H^{\ol{G}}_n(E_{\all}\ol{G},E_{\ol{\fac}}\ol{G};\bK_R).
\end{gather*}
Here $f_A, \eta_A$ and $f_B, \eta_B$ are inclusion-induced maps.

\item The maps
\begin{gather*}
\eta_A \circ \sigma_A  : \wt{\Nil}_{n-1}(R[F];R[G_1-F],R[G_2-F]) ~\xrightarrow{\cong}~ H^{G}_n(E_{\all}G,E_{\fac}G;\bK_R) \\
\eta_B \circ \sigma_B :  \wt{\Nil}_{n-1}(R[F],\alpha) \oplus  \wt{\Nil}_{n-1}(R[F],\alpha\inv)~\xrightarrow{\cong}~ H^{\ol{G}}_n(E_{\all}\ol{G},E_{\ol{\fac}}\ol{G};\bK_R)
\end{gather*}
are isomorphisms where $\sigma_A$ and $\sigma_B$ are Waldhausen's split injections.
\end{enumerate}
\end{lem}

The statement of Lemma 3.1 of \cite{DQR} does not explicitly identify the isomorphisms in Part (2) above, but the identification follows from the last
 paragraph of the proof.

It is not difficult to compute
$H^{G}_n(E_{\fac}G;\bK_R)$ and $H^{\ol{G}}_n(E_{\ol{\fac}}\ol{G};\bK_R)$ in terms of a Wang sequence and a
Mayer--Vietoris sequence respectively. An example is in \fullref{Subsec_PSL2Z}.

Next, we further assume $\ol{G} \subset G$ with $|G:\ol{G}|=2$.  Then $C_2 = G/ \ol{G}$ acts on $K_n(R\ol{G}) = H^{\ol{G}}_n(E_{\all} \ol{G}; \bK_R)$ by conjugation.
By \cite[Remark~3.21]{DQR}, there is a $C_2$-action on $H_n^{\ol{G}}(E_\all\ol{G}, E_{\ol{\fac}}\ol{G};\bK_R)$ so that $\eta_B$ and $\theta_{!!}$ below are $C_2$-equivariant.

\begin{lem}\label{lem:DKR_fbc}  Let $n \leq 1$ be an integer.
There is a commutative diagram of $C_2$-equivariant homomorphisms:
\[\begin{diagram}
\node{\wt{\Nil}_{n-1}(R[F],\alpha)\oplus \wt{\Nil}_{n-1}(R[F],\alpha\inv)}
\arrow{s,b}{\begin{pmatrix} i_* & \tau_A\inv i'_* \beta_u^+ \end{pmatrix}}
\arrow{e,t}{\di{\eta_B \circ \sigma_B}}
\node{H_n^{\ol{G}}(E_\all\ol{G}, E_{\ol{\fac}}\ol{G}; \bK_R)}
\arrow{s,b}{\di{\theta_{!!}}}\\
\node{\wt{\Nil}_{n-1}(R[F];\fB_1,\fB_2)}
\arrow{e,t}{\di{\eta_A \circ \sigma_A}}
\node{H_n^G(E_\all G, E_\fac G; \bK_R)}
\end{diagram}\]
Here, the $C_2 =G/\ol{G}$-action on the upper left-hand corner is given in \fullref{prop:induction}, on the upper right it is given by \cite[Remark~3.21]{DQR}, and on each lower corner it is trivial.
\end{lem}

\begin{proof}
This follows from \fullref{prop:induction} and the $C_2$-equivariance of $\eta_B$.
\end{proof}

Recall that if $C_2 = \{1,T\}$ and if $M$ is a $\Z[C_2]$-module
then the \textbf{coinvariant group} $M_{C_2} = H_0(C_2 ;M)$ is the
quotient group of $M$ modulo the subgroup  $\{m - Tm \mid m \in
M\}$.

\begin{lem}\label{lem:DKR_induction}  Let $n \leq 1$ be an integer.
There is an induction-induced isomorphism:
\[
\left(H_n^{\ol{G}}(E_\all\ol{G}, E_{\ol{\fac}}\ol{G}; \bK_R)\right)_{C_2}
~\longrightarrow~
H_n^G(E_{\fac \, G \cup \sub\, \ol{G}} G, E_{\fac }G; \bK_R).
\]
\end{lem}

\begin{proof}
Recall $G/\ol{G} = C_2$.  Since $\ol{\fac} = \fac \cap \ol{G}$, by \cite[Lemma~4.1(i)]{DQR} there is a identification of $\Z[C_2]$-modules
\[
H_n^{\ol{G}}(E_\all\ol{G}, E_{\ol{\fac}}\ol{G}; \bK)
~=~
\pi_n(\bK/\bK_\fac)(G/\ol{G}).
\]
The $C_2$-coinvariants can be interpreted as a $C_2$-homology group:
\[
\left(\pi_n(\bK/\bK_\fac)(C_2)\right)_{C_2}
~=~
H_0^{C_2}(EC_2; \pi_n(\bK/\bK_\fac)(C_2)).
\]
By \fullref{DW}(2) and \fullref{lem:DKR_fbc}, the coefficient $\Z[C_2]$-module is induced from a $\Z$-module.  By the Atiyah--Hirzebruch spectral sequence (which collapses at $E^2$), note
\[
H_0^{C_2}(EC_2; \pi_n(\bK/\bK_\fac)(C_2))
~=~
H_n^{C_2}(EC_2; (\bK/\bK_\fac)(C_2)).
\]
Therefore, by \cite[Lemma~4.6, Lemma 4.4, Lemma 4.1]{DQR}, we conclude:
\begin{eqnarray*}
H_n^{C_2}(EC_2; (\bK/\bK_\fac)(C_2))
&=&
H_n^G(E_{\sub\ol{G}}G; \bK/\bK_\fac)\\
&=&
H_n^G(E_{\fac \, G \cup \sub\, \ol{G}}G; \bK/\bK_\fac)\\
&=&
H_n^G(E_{\fac \, G \cup \sub\, \ol{G}}G, E_{\fac}G; \bK).
\end{eqnarray*}
\end{proof}

The identifications in the above proof are extracted from the proof of  \cite[Theorem~1.5]{DQR}.

\begin{proof}[Proof of \fullref{isom_conj}]
Let $\varphi :\, \g \to G$ be a homomorphism of groups.  Using the Transitivity Principle applied to the families
$\varphi^*\fbc \subset \varphi^*\vc$, it suffices to show that
\[
H^H_n(E_{\all} H, E_{\varphi^*\fbc | H}H; \bK_R) = 0
\]
for all $n \leq 1$ and for all $H \in \varphi^*\vc - \varphi^* \fbc$.  To identify the family $\varphi^*\fbc | H$ we will
use two facts, the proofs of which are left to the reader.

\begin{itemize}
\item If $q : A \to B$ is a group epimorphism with finite kernel, then $\fbc \, A = q^* \fbc \, B$.
(The key step is to show that an epimorphic image of a finite-by-cyclic group is finite-by-cyclic.)


\item  If $q : A \to B = G_1 *_F G_2$ is an group epimorphism, then $A = q^{-1}G_1 *_{q^{-1}F} q^{-1}G_2$ and
$\fac \, A = q^*\fac \, B$.
\end{itemize}

Let $\varphi|:\, H \to \varphi(H)$ denote the restriction of $\varphi$ to $H$.  By the definition of both sides,
\[
\varphi^*\fbc | H ~=~ \varphi|^*(\fbc \, \varphi(H)).
\]
Since $H \in \varphi^* \vc - \varphi^* \fbc$, we have $\varphi(H) \in \vc - \fbc$. So, by \fullref{lem_vctypes}, there is
an epimorphism $p:\, \varphi(H) \to D_{\infty} = \Z_2 *_1 \Z_2$ with finite kernel.  By the first fact above
\[
\varphi|^*(\fbc \, \varphi(H)) ~=~ \varphi|^*(p^*(\fbc \, D_\infty)).
\]
Next, write $\ol{H} := (p \circ \varphi|)^{-1} (\Z)$. Note
\begin{eqnarray*}
\varphi|^*(p^*(\fbc \, D_\infty)) &=& (p \circ \varphi|)^*(\fbc \, D_\infty) \\
&=& (p \circ \varphi|)^*(\fac \, D_\infty \cup \sub \, \Z)\\
&=& (p \circ \varphi|)^*(\fac \, D_\infty ) \cup  (p \circ \varphi|)^*(\sub \, \Z)\\
&=& \fac \, H \cup \sub \, \ol{H},
\end{eqnarray*}
where the last equality uses the second fact above.
Thus it suffices to prove, for any group $H$ mapping epimorphically to $D_\infty$ and for all $n \leq 1$, that
\[
H^H_n(E_{\all} H, E_{\fac \, H \cup \sub \, \ol{H}}; \bK_R) = 0
\]
where $\ol{H}$ is the inverse image of the maximal infinite cyclic subgroup of $D_\infty$.

Consider the following composite
\begin{multline*}
\left(H_n^{\ol{H}}(E_{\all}\ol{H}, E_{\ol{\fac}}\ol{H};\bK_R)\right)_{H/\ol{H}} ~\xrightarrow{~\alpha~}~ H^H_n( E_{\fac \, H \cup \sub \, \ol{H}}H,E_{\fac} H; \bK_R) \\
~\xrightarrow{~\beta~}~ H^H_n(E_{\all} H, E_\fac H; \bK_R).
\end{multline*}
The map $\alpha$ exists and is an isomorphism by \fullref{lem:DKR_induction}.
Apply $C_2$-covariants to the commutative diagram in the statement of \fullref{lem:DKR_fbc}.  In this diagram of $C_2$-coinvariants,
the top and bottom are isomorphisms  by \fullref{DW}(2)  and the left map is an isomorphism by \fullref{prop:induction} and \fullref{maink}.    Hence the right-hand map, which is $\beta \circ \alpha$, is an isomorphism  for all $n \leq 1$.  It follows that $\beta$ is an isomorphism  for all $n \leq 1$.
So, by the exact sequence of a triple, we obtain
\[
H_n^H(E_{\all}H, E_{\fac \, H \cup \sub\, \ol{H}}H; \bK_R) ~=~ 0
\]
for all $n \leq 1$ as desired.
\end{proof}

\subsection{$K$-theory of the modular group}\label{Subsec_PSL2Z}

Let $\g = \Z_2 * \Z_3 = PSL_2(\Z)$.  The following theorem follows
from applying our main theorem and  the recent proof \cite{BLR} of
the Farrell--Jones conjecture in $K$-theory for word hyperbolic
groups.

The Cayley graph for  $\Z_2 * \Z_3$ with respect to the generating
set given by the nonzero elements of $\Z_2$ and $\Z_3$ has the
quasi-isometry type of the usual Bass--Serre tree for the
amalgamated product (\fullref{Fig_Tree}).
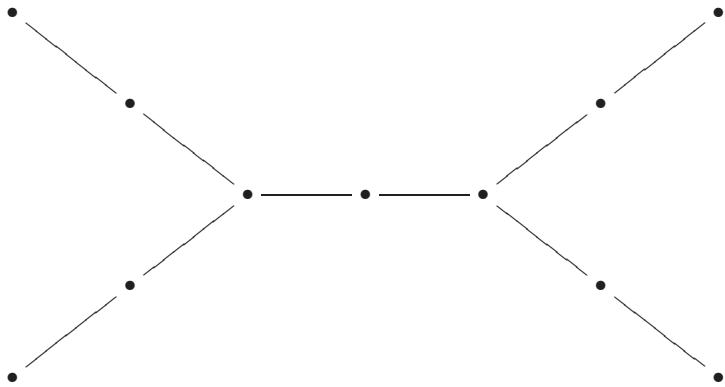
\begin{figure}[!ht]\begin{center}$\xymatrix@C+10pt{
\bullet \ar@{-}[dr] & & & &  & & \bullet \\
& \bullet \ar@{-}[dr] & & & & \bullet \ar@{-}[ur] & \\
& & \bullet \ar@{-}[r] &  \bullet \ar@{-}[r] & \bullet \ar@{-}[ur] \ar@{-}[dr]& &\\
& \bullet \ar@{-}[ur]  & & & & \bullet \ar@{-}[dr] &\\
\bullet \ar@{-}[ur] & & & & &  & \bullet }$\end{center}
\caption{Bass--Serre tree for $PSL_2(\Z)$}\label{Fig_Tree}
\end{figure}
This is an infinite tree with alternating vertices of valence two
and three. The group $\g$ acts on the tree, with the generator of
order two acting by reflection through an valence two vertex and the
generator of order three acting by rotation through an adjoining
vertex of valence three.

Any geodesic triangle in the Bass--Serre tree has the property
that the union of two sides is the union of all three sides.  It
follows that the Bass--Serre graph is $\delta$-hyperbolic for any
$\delta > 0$, the Cayley graph is $\delta$-hyperbolic for some
$\delta > 0$, and hence $\g$ is a hyperbolic group.

\begin{thm} \label{Jim}
For any ring $R$ and integer $n$,
\begin{multline*}
K_n(R[\g]) = ( K_n(R[\Z_2]) \oplus K_n(R[\Z_3]))/K_n(R) \\
\oplus \bigoplus_{\cM_C}  \wt {\Nil}_{n-1}(R) \oplus
\wt {\Nil}_{n-1}(R) \oplus \bigoplus_{\cM_D} \wt
{\Nil}_{n-1}(R)
\end{multline*}
where $\cM_C$ and $\cM_D$ are the set of conjugacy classes of
maximal infinite cyclic subgroups and maximal infinite dihedral
subgroups, respectively. Moreover, all virtually cyclic subgroups
of $\g$ are cyclic or infinite dihedral.
\end{thm}

\begin{proof}
By \fullref{DW}, the exact sequence of
$(E_\all\g,E_\fin\g)$ is short exact and split:
\[
H_n^{\g}(E_\fin\g; \bK_R) \to H_n^{\g}(E_\all\g; \bK_R) \to H_n^{\g}(E_\all\g, E_\fin\g; \bK_R).
\]
Then, by the Farrell--Jones Conjecture~\cite{BLR} for word
hyperbolic groups, we obtain
\begin{eqnarray*}
K_n(R[\g]) &=& H_n^{\g}(E_\fin\g; \bK_R) \oplus H_n^{\g}(E_\all\g, E_\fin\g; \bK_R)\\
\label{vc,fin}  &=& H_n^{\g}(E_\fin\g; \bK_R) \oplus H_n^{\g}(E_\vc\g, E_\fin\g; \bK_R).
\end{eqnarray*}

Observe $E_\fin\g$ is constructed as a pushout of $\g$-spaces
\[
\begin{CD}
\g \sqcup \g @>>> \g/\Z_2 \sqcup \g/\Z_3\\
@VVV @VVV\\
\g \times D^1 @>>> E_\fin\g.
\end{CD}
\]
Then $E_\fin\g$ is the Bass--Serre tree for $\g = \Z_2 * \Z_3$.
Note that $H^{\g}_*(\g/H; \bK_R) = K_*(R[H])$.  The pushout gives,
after canceling a $K_n(R)$ term, a split long exact sequence
\[
\cdots \to K_n(R) \to K_n(R[\Z_2]) \oplus K_n(R[\Z_3]) \to H_n^{\g}(E_{\fin}\g; \bK_R) \to K_{n-1}(R) \to \cdots.
\]
Hence
\[
H_n^{\g}(E_\fin\g; \bK_R) ~=~  ( K_n(R[\Z_2]) \oplus K_n(R[\Z_3]))/K_n(R).
\]

Next, for a word hyperbolic group $G$,
\[
H^G_n(E_{\vc}G,E_{\fin}G; \bK) \cong \bigoplus_{[V] \in \cM(G)} H^V_n(E_{\vc}V,E_{\fin}V; \bK)
\]
where $\cM(G)$ is the set of conjugacy classes of maximal
virtually cyclic subgroups of $G$ (see \cite[Theorem~8.11]{Lueck} and \cite{Juan-PinedaLeary}).  The geometric
interpretation of this result is that $E_{\vc}G$ is
obtained by coning off each geodesic in the tree $E_{\fin}G$; then
apply excision.

The Kurosh subgroup theorem implies that a subgroup of $\Z_2 *
\Z_3$ is a free product of $\Z_2$'s, $\Z_3$'s, and $\Z$'s.  Note
that $\Z_2 * \Z_3 = \langle a,b ~|~ a^2 = 1 = b^3 \rangle$,  $\Z_3 *
\Z_3 = \langle c,d ~|~ c^3 = 1 = d^3 \rangle$, and  $\Z_2 * \Z_2 *
\Z_2 = \langle e,f,g ~|~ e^2 = f^2 = g^2= 1 \rangle$ have free
subgroups of rank 2, for example $\langle ab,ab^2\rangle$, $\langle
cd, cd^2\rangle$, and $\langle ef, fg \rangle$.  On the other hand,
the free group $F_2$ rank 2  is not a virtually cyclic group since
its first Betti number $\beta_1(F_2) = \text{rank}~H_1(F_2) =
2$, while for a virtually cyclic group $V$, transferring to the
cyclic subgroup $C\subset V$ of finite index shows that
$\beta_1(V)$ is 0 or 1. Subgroups of virtually cyclic groups are
also virtually cyclic. Therefore all virtually cyclic subgroups of
$\g$ are cyclic or infinite dihedral.

By the fundamental theorem of $K$-theory and Waldhausen's Theorem
(\ref{DW}):
\begin{eqnarray*}
H_n^{\Z}(E_{\vc}\Z, E_{\fin}\Z; \bK_R) &=&  \wt {\Nil}_{n-1}(R) \oplus  \wt {\Nil}_{n-1}(R) \\
H_n^{D_{\infty}}(E_{\vc}D_{\infty}, E_{\fin}D_{\infty}; \bK_R) &=&   \wt \Nil_{n-1}(R;R,R)
\end{eqnarray*}

Finally, by \fullref{dcor}(1), we obtain exactly one type of
Nil-group:
\[
\wt{\Nil}_{n-1}(R;R,R) ~\cong~ \wt{\Nil}_{n-1}(R).
\proved\]
\end{proof}

\begin{rem}
The sets $\cM_C$ and $\cM_D$ are countably infinite.  This can be
shown by parameterizing these subsets either: combinatorially (using
that elements in $\g$ are words in $a,b,b^2$), geometrically
(maximal virtually cyclic subgroups correspond to stabilizers of
geodesics in the Bass--Serre tree $E_{\fin}\g$, where the geodesic
may or may not be invariant under an element of order 2), or number
theoretically (using solutions to Pell's equation and Gauss' theory
of binary quadratic forms \cite{S}).
\end{rem}

Let us give an overview and history of some related work. The
Farrell--Jones Conjecture and the classification of virtually cyclic
groups (\ref{lem_vctypes}) focused attention on the
algebraic $K$-theory of groups mapping to the infinite dihedral
group.  Several years ago  James Davis and Bogdan Vajiac  outlined a
unpublished proof of \fullref{first} when $n \leqslant 0$ using
controlled topology and hyperbolic geometry. Lafont and Ortiz~\cite{LafontOrtiz} proved that $\wt{\Nil}_n(\Z[F];\Z[V_1-
F],\Z[V_2- F])=0$ if and only if
$\wt{\Nil}_n(\Z[F],\alpha)=0$ for any virtually cyclic group
$V$ with an epimorphism $V \to D_{\infty}$ and $n=0,1$. More
recently, Lafont--Ortiz~\cite{LO2} have studied the more general
case of the $K$-theory $K_n(R[G_1 *_F G_2])$ of an injective
amalgam, where $F, G_1, G_2$ are finite groups.  Finally, we
mentioned the paper \cite{DQR}, which was written in parallel with
this one; it an alternate proof of \fullref{first}.
Also, \cite{DQR} provides several auxiliary results used in \fullref{Third} of
this paper. The Nil-Nil isomorphism of \fullref{first}
has been used in a geometrically
motivated computation of Lafont--Ortiz~\cite[Section~6.4]{LO_simplex}.


\section{Codimension 1 splitting and semi-splitting}\label{codim1}

We shall now give a topological interpretation of the
Nil-Nil \fullref{Thm_HigherDKR}, proving in \fullref{Thm_TopSemisplit}
that every homotopy equivalence of finite CW-complexes $f:M \to X=X_1\cup_YX_2$ with $X_1,X_2,Y$ connected and $\pi_1(Y) \to \pi_1(X)$ injective is ``semi-split'' along $Y \subset X$, assuming that $\pi_1(Y)$ is of finite index in $\pi_1(X_2)$.
Indeed, the proof of \fullref{Thm_HigherDKR} is motivated by the codimension 1 splitting obstruction theory of Waldhausen~\cite{Waldhausen_1969}, and the subsequent algebraic $K$-theory decomposition theorems of
Waldhausen~\cite{Waldhausen_1973,Waldhausen_Rings}.
The papers \cite{Waldhausen_1969,Waldhausen_1973}  developed both an algebraic splitting obstruction theory for chain complexes over injective generalized
free products, and a geometric codimension 1 splitting obstruction theory; the geometric splitting obstruction is the algebraic splitting obstruction of the cellular chain complex. There are
parallel theories for the separating type (A) (amalgamated free product) and
the non-separating type (B) (HNN extension). We first briefly outline the theory, mainly for type (A).

The cellular chain complex of the universal
cover $\wt{X}$ of a connected CW-complex $X$ is a based
free $\Z[\pi_1(X)]$-module chain complex $C(\wt{X})$ such that
$H_*(\wt{X})=H_*(C(\wt{X}))$.
The \textbf{kernel $\Z[\pi_1(X)]$-modules} of a map  $f:M \to X$
are defined by
\[
K_*(M)~:=~H_{*+1}(\wt{f}:\wt{M} \to \wt{X})
\]
with $\wt{M}:=f^*\wt{X}$ the pullback cover of $M$ and
$\wt{f}:\wt{M} \to \wt{X}$ a
$\pi_1(X)$-equivariant lift of $f$. For a cellular map $f$ of CW-complexes let
\[
\fK(M)~:=~\fC(\wt{f}:C(\wt{M}) \to C(\wt{X}))_{*+1}
\]
be the algebraic mapping cone of the induced $\Z[\pi_1(X)]$-module chain map $\wt{f}$,
with homology $\Z[\pi_1(X)]$-modules
$$H_*(\fK(M))~=~K_*(M)~=~H_{*+1}(\wt{f}:\wt{M} \to \wt{X})~.$$
For $n \geqslant 1$ the map $f:M \to X$ is $n$-connected
if and only if $f_*:\pi_1(M) \cong \pi_1(X)$ and $K_r(M)=0$ for $r < n$,
in which case the Hurewicz map is an isomorphism:
\[
\pi_{n+1}(f)~=~\pi_{n+1}(\wt{f}) \to K_n(M)~=~H_{n+1}(\wt{f})~.
\]
By the theorem of J\,H\,C~Whitehead,
$f:M \to X$ is a homotopy equivalence if and only if
$f_*:\pi_1(M) \cong \pi_1(X)$ and $K_*(M)=0$ (if and only if
$\fK(M)$ is chain contractible).

Decompose the boundary of the $(n+1)$-disk as a union of upper and
lower $n$-disks:
\[
\partial D^{n+1}~=~S^n~=~D^n_+ \cup_{S^{n-1}} D^n_-~.
\]
Given a CW-complex $M$ and a cellular
map $\phi:D^n_+ \to M$ define a new CW-complex
\[
M'~=~(M \cup_{\partial \phi} D^n_-)\cup_{\phi \cup 1}D^{n+1}
\]
by attaching an $n$-cell and an $(n+1)$-cell, with
\[
\partial \phi~=~\phi\vert~:~S^{n-1} \to M~,~
\phi \cup 1~:~S^n~=~D^n_+ \cup_{S^{n-1}} D^n_- \to M
\cup_{\partial \phi}D^n_-~.
\]
The inclusion $M \subset M'$ is a homotopy equivalence called an \textbf{elementary expansion}. The
cellular based free $\Z[\pi_1(M)]$-module chain complexes fit into
a short exact sequence
\[
0 \to C(\wt{M}) \to C(\wt{M}') \to C(\wt{M}',\wt{M}) \to 0
\]
with
\[
\xymatrix{C(\wt{M}',\wt{M})~:~\cdots \ar[r] & 0
\ar[r] & \Z[\pi_1(M)] \ar[r]^-{\di{1}} & \Z[\pi_1(M)] \ar[r] & 0 \ar[r] & \cdots}
\]
concentrated in dimensions $n,n+1$. For a commutative diagram of cellular maps
\[
\xymatrix{D^n_+ \ar[r]^-{\di{\phi}} \ar[d] & M \ar[d]^-{\di{f}} \\
D^{n+1} \ar[r]^-{\di{\delta\phi}} & X}
\]
$f$ extends to a cellular map
\[
f'~=~(f\cup \delta\phi\vert_{D^n_-}) \cup \delta\phi~:~M'~=~(M \cup_{\partial \phi} D^n_-)\cup_{\phi \cup 1}D^{n+1} \to X
\]
which is also called an \textbf{elementary expansion}, and
there is defined a short exact sequence
of based free $\Z[\pi_1(X)]$-module chain complexes
\[
0 \to \fK(M) \to \fK(M') \to C(\wt{M}',\wt{M}) \to 0~.
\]

Recall the \textbf{Whitehead group} of a group $G$ is defined by
\[
\Wh(G)~:=~K_1(\Z[G])/\{\pm g \mid g\in G\}~.
\]
Suppose the CW-complexes $M,M',X$ are finite. The \textbf{Whitehead torsion} of a homotopy
equivalence $f:M \to X$ is
\[
\tau(f)~=~\tau(\fK(M)) \in \Wh(\pi_1(X))~.
\]
Homotopy equivalences $f:M \to X$ and
$f':M' \to X$ are \textbf{simple-homotopic} if
\[
\tau(f)~=~\tau(f') \in \Wh(\pi_1(X))~.
\]
This is equivalent to being able to obtain $f'$ from $f$ by a
finite sequence of elementary expansions and subdivisions and
their formal inverses. For details, see Cohen's book~\cite{Cohen_book}.

A \textbf{2-sided codimension 1 pair} $(X,Y \subset X)$ is a pair of spaces
such that the inclusion $Y=Y \times \{0\}\subset X$ extends to an
open embedding $Y \times \R \subset X$.
We say that a homotopy equivalence $f:M \to X$  \textbf{splits along $Y \subset
X$} if  the restrictions $f\vert:N=f^{-1}(Y) \to Y$, $f\vert:M - N \to X - Y$ are also homotopy equivalences.

In dealing with maps $f:M \to X$ and 2-sided codimension 1 pairs $(X,Y)$
we shall assume that $f$ is cellular and that both $(X,Y)$ and
$(M,N=f^{-1}(Y))$ are a 2-sided codimension 1 CW-pair.

A  2-sided codimension 1 CW-pair $(X,Y)$ is \textbf{$\pi_1$-injective} if $X,Y$ are connected and $\pi_1(Y) \to \pi_1(X)$ is injective.  As usual, there are two cases, according as to whether $Y$ separates $X$ or not:
\begin{enumerate}
\item[(A)] The \textbf{separating} case: $X- Y$ is disconnected, so
\[
X~=~X_1\cup_Y X_2
\]
with $X_1,X_2$ connected. By the Seifert-van Kampen theorem $$\pi_1(X)~=~\pi_1(X_1)*_{\pi_1(Y)}\pi_1(X_2)$$
 is the amalgamated free product determined by the injections $i_k:\pi_1(Y) \to \pi_1(X_k)$ ($k=1,2$).
\item[(B)] The \textbf{non-separating} case: $X- Y$ is connected, so
\[
X~=~X_1/\{y \sim ty \mid y \in Y\}
\]
for a connected space $X_1$ (a deformation retract of $X-Y$)
which contains two disjoint copies $Y \sqcup tY \subset X_1$
of $Y$. By the Seifert-van Kampen theorem
$$\pi_1(X)~=~\pi_1(X_1)*_{i_1,i_2}\{t\}$$
is the HNN extension determined by the injections $i_1,i_2:\pi_1(Y) \to \pi_1(X_1)$, with
$i_1(y)t=ti_2(y)$ ($y \in \pi_1(Y)$).
\end{enumerate}

\begin{rem} Let
$\wt{X}$ be the universal cover of $X$, and let $\ol{X}:=\wt{X}/\pi_1(Y)$, so that
for both types (A) and (B),  $(\ol{X},Y)$ is a $\pi_1$-injective 2-sided codimension 1 pair of the separating type (A), with $\ol{X}=\ol{X}^-\cup_Y\ol{X}^+$ for connected subspaces $\ol{X}^-,\ol{X}^+\subset \ol{X}$ such that
$$\pi_1(\ol{X})~=~\pi_1(\ol{X}^-)~=~\pi_1(\ol{X}^+)~=~\pi_1(Y)~.$$
Moreover for type (B), when $i_1,i_2$ are isomorphisms, the HNN extension simplifies to:
\[\begin{CD}
1 @>>> \pi_1(Y) @>>> \pi_1(X)~=~\pi_1(Y)\rtimes_{\alpha} \Z  @>>> \Z @>>> 1
\end{CD}\]
with automorphism $\alpha=(i_1)^{-1}i_2$ of $\pi_1(Y)$, studied originally by Farrell and Hsiang \cite{FarrellHsiang}.
\end{rem}

\emph{From now on, we shall only consider the separating case (A) of $X=X_1 \cup_Y X_2$.} Write
\begin{gather*}
\pi_1(X)~=~G~,~\pi_1(X_1)~=~G_1~,~\pi_1(X_2)~=~G_2~,~
\pi_1(Y)~=~H~,\\[1ex]
i_k~:~\Z[H] \to \Z[G_k]~=~\Z[H]\oplus \fB_k~,~\fB_k~=~\Z[G_k-H]
\end{gather*}
with $\fB_k$ free as both a right and a left $\Z[H]$-module, and
$$\Z[G]~=~\Z[G_1]*_{\Z[H]}\Z[G_2]~=~
\Z[H] \oplus \fB_1 \oplus \fB_2 \oplus \fB_1\fB_2  \oplus \fB_2\fB_1 \oplus \dots~.$$
Use the injections $i_k:H \to G_k$ to define covers
\[\ol{X}_1~=~\wt{X}_1/H \subset \ol{X}^-~,~
\ol{X}_2~=~\wt{X}_2/H \subset \ol{X}^+\]
such that $\ol{X}_1 \cap \ol{X}_2 = Y$ and
\begin{eqnarray*}
\wt{X} &=& \left(\bigcup\limits_{g_1 G_1 \in G/G_1}g_1\wt{X}_1\right) \cup
_{\left(\bigcup\limits_{h H \in G/H}h\wt{Y}\right)}
\left(\bigcup\limits_{g_2 G_2 \in G/G_2}g_2\wt{X}_2\right)\\
\ol{X} &=& \left(\bigcup\limits_{g_1 G_1 \in G/G_1}g_1\ol{X}_1\right) \cup
_{\left(\bigcup\limits_{h H \in G/H}hY\right)}
\left(\bigcup\limits_{g_2 G_2 \in G/G_2}g_2\ol{X}_2\right)
\end{eqnarray*}
with  $\wt{X}_k$ the universal cover of $X_k$, and $\wt{Y}$ the universal cover of $Y$.

Let $(f,g):(M,N) \to (X,Y)$ be a map of separating $\pi_1$-injective codimension 1 finite CW-pairs. This gives an exact sequence of based  free $\Z[H]$-module chain complexes
\begin{equation}\label{eqn:chainsplitting}\begin{CD}
0 @>>> \fK(N) @>>> \fK(\ol{M}) @>>> \fK(\ol{M}^-,N)\oplus \fK(\ol{M}^+,N) @>>> 0
\end{CD}\end{equation}
inducing a long exact sequence of homology modules
\[\xymatrix{
\cdots \ar[r] & K_r(N) \ar[r] & K_r(\ol{M}) \ar[r] & K_r(\ol{M}^+,N)\oplus K_r(\ol{M}^-,N) \ar[r] & K_{r-1}(N) \ar[r] & \cdots~.
}\]
Note that $f:M \to X$ is a homotopy equivalence if and
only if $f_*:\pi_1(M) \to \pi_1(X)$ is an isomorphism
and $\fK(\ol{M})$ is contractible. The map of pairs $(f,g):(M,N) \to (X,Y)$ is a split homotopy equivalence if and
only if any two of the chain complexes in \eqref{eqn:chainsplitting} are contractible, in which case the third chain complex is
also contractible.

Suppose $f:M \to X$ is a homotopy equivalence. Then
\[
K_*(\ol{M})~=~K_*(M)~=~0~,~ K_*(N)~=~K_{*+1}(\ol{M}^-,N) \oplus K_{*+1}(\ol{M}^+,N)~.
\]
We obtain an exact sequence of $\Z[H]$-module chain complexes
\begin{gather*}
0 \to \fK(\ol{M}_1,N) \to \fK(\ol{M}^-,N) \xrightarrow{~\di{\rho_1}~} \fK(\ol{M}^-,\ol{M}_1)~=~ \fB_1\otimes_{\Z[H]}\fK(\ol{M}^+,N) \to 0~\\
0 \to \fK(\ol{M}_2,N) \to \fK(\ol{M}^+,N) \xrightarrow{~\di{\rho_2}~} \fK(\ol{M}^+,\ol{M}_2)~=~ \fB_2\otimes_{\Z[H]}\fK(\ol{M}^-,N) \to 0~.
\end{gather*}
The pair $(\rho_1,\rho_2)$ of intertwined chain maps is \emph{chain homotopy nilpotent}, in the sense that the following chain map is a $\Z[G]$-module chain equivalence:
\begin{multline*}
\begin{pmatrix} 1 & \rho_2 \\ \rho_1 & 1 \end{pmatrix}~:~
\Z[G]\otimes_{\Z[H]}(\fK(\ol{M}^-,N)\oplus\fK(\ol{M}^+,N))\\[1ex]
\longrightarrow~ \Z[G]\otimes_{\Z[H]}(\fK(\ol{M}^-,N)\oplus\fK(\ol{M}^+,N))~.
\end{multline*}

\begin{dfn}
Let $x=(P_1,P_2,\rho_1,\rho_2)$ be an object of $\NIL(\Z[H];\fB_1,\fB_2)$.
\begin{enumerate}
\item
Let $x'=(P'_1,P'_2,\rho'_1,\rho'_2)$ be another object. We say $x$ and $x'$ are \textbf{equivalent} if
\[
[P_1] = [P'_1]~,~[P_2] = [P'_2] ~\in~ \wt{K}_0(\Z[H])~,~[x] = [x'] ~\in~ \wt{\Nil}_0(\Z[H];\fB_1,\fB_2)~,
\]
or equivalently
\begin{multline*}
[x']-[x] \in K_0(\Z)\oplus K_0(\Z) \subseteq \Nil_0(\Z[H];\fB_1,\fB_2)\\[1ex]
=~K_0(\Z[H])\oplus K_0(\Z[H]) \oplus \wt{\Nil}_0(\Z[H];\fB_1,\fB_2)
\end{multline*}
with $K_0(\Z) \oplus K_0(\Z)$ the subgroup generated by $(\Z[H],0,0)$ and $(0,\Z[H],0,0)$.
\item
Let $k=1$ or $2$. Let $y_k \in \mathrm{ker}(\rho_k)$ generate
a direct summand $\langle y_k \rangle \subseteq P_k$. Define an object $x'$ in $\NIL(\Z[H];\fB_1,\fB_2)$ by
$$x'~=~(P'_1,P'_2,\rho'_1,\rho'_2)~=~\begin{cases} (P_1/\langle y_1 \rangle,P_2,[\rho_1],[\rho_2])&{\rm if}~k=1 \\
(P_1,P_2/\langle y_2 \rangle,[\rho_1],[\rho_2])&{\rm if}~k=2
\end{cases}$$
with an exact sequence  in $\NIL(\Z[H];\fB_1,\fB_2)$
$$\begin{cases}
\xymatrix@C+10pt{0 \ar[r] & (\Z[H],0,0,0) \ar[r]^-{\di{(y_1,0)}} & x \ar[r] &x' \ar[r] &  0}\\[1ex]
\xymatrix@C+10pt{0 \ar[r] & (0,\Z[H],0,0) \ar[r]^-{\di{(0,y_2)}} & x \ar[r] &x' \ar[r] &  0}~.
\end{cases}$$
Thus $x'$ is equivalent to $x$, obtained by the \textbf{algebraic cell-exchange}
which \textbf{kills} $y_k \in P_1 \oplus P_2$.
\end{enumerate}
\end{dfn}

It can be shown that two objects $x$ and $x'$ in $\NIL(\Z[H];\fB_1,\fB_2)$ are equivalent if and only if $x'$ can be obtained from $x$ by a finite sequence of isomorphisms, algebraic cell-exchanges, and their formal inverses.

Geometric cell-exchanges (called \textbf{surgeries} in \cite{Waldhausen_1969}) determine algebraic cell-exchanges.
In the highly-connected case, algebraic and geometric cell-exchanges occur in tandem:

\begin{thm} [\cite{Waldhausen_1969}]\label{W}
Let $(f,g):(M,N) \to (X,Y)$ be a map of separating $\pi_1$-injective codimension 1 finite CW-pairs,
with $f:M \to X$ a homotopy equivalence. Write $X=X_1\cup_YX_2$ with induced amalgam $\pi_1(X)=G=G_1*_HG_2$ of fundamental groups.

\noindent\textbf{(i)}
Let $k=1,2$. Suppose for some $n \geqslant 0$ that we are given a map
\[
(\phi,\partial \phi)~:~(D^{n+1},S^n) ~\longrightarrow~ (M_k,N)
\]
and a null-homotopy of pairs
\[
(\theta,\partial \theta)~:~(f\vert_{M_k} \circ \phi, g \circ \partial \phi)~\simeq~(*,*)~:~(D^{n+1},S^n) ~\longrightarrow~	 (X_k,Y)~.
\]
Assume they represent an element in $\mathrm{ker}(\rho_k)$ (with $\epsilon=-$ if $k=1$; $\epsilon=+$ if $k=2$):
\begin{multline*}
y_k~=~[\phi,\theta] \in \mathrm{im}(K_{n+1}(\ol{M}_k,N) \to K_{n+1}(\ol{M}^{\epsilon},N))\\[1ex]
=~\mathrm{ker}(\rho_k:K_{n+1}(\ol{M}^{\epsilon},N) \to
\fB_k \otimes_{\Z[H]}K_{n+1}(\ol{M}^{-\epsilon},N)) \subseteq K_n(N)~.
\end{multline*}
The map $(f,g)$ extends to the map of codimension 1 pairs
\begin{multline*}
(f',g')~:=~((f \cup f\vert_{M_k}\circ \phi)\cup \theta, g \cup \partial \theta)~:\\[1ex]
(M',N')~:=~((M\cup_{\partial \phi}D^{n+1}) \cup_{\phi \cup 1} D^{n+2},N \cup_{\partial \phi}D^{n+1}) ~\longrightarrow~ (X,Y)
\end{multline*}
where the new $(n+2)$-cell has attaching map
\[
\phi \cup 1~:~\partial D^{n+2}~=~D^{n+1} \cup_{S^n}D^{n+1} ~\longrightarrow~ M \cup_{\partial \phi}D^{n+1}~.
\]
The homological effect on $(f,g)$ of this \textbf{geometric cell-exchange} is no change in
\[\begin{cases}
K_r(\ol{M}'^{\epsilon},N') ~=~ K_r(M^{\epsilon},N) & \text{for}~ r\neq n+1,n+2~,\\[1ex]
K_r(\ol{M}'^{-\epsilon},N') ~=~ K_r(M^{-\epsilon},N) & \text{for all}~ r \in \Z
\end{cases}\]
except there is a five-term exact sequence
\[\begin{CD}
0 @>>> K_{n+2}(\ol{M}^\epsilon,N) @>>>
K_{n+2}(\ol{M}'^\epsilon,N')
@>>> \Z[H]\\[1ex]
@>{\di{y_k}}>> K_{n+1}(\ol{M}^\epsilon,N)
@>>> K_{n+1}(\ol{M}'^\epsilon,N') @>>> 0~.
\end{CD}\]
The inclusion $h:M \subset M'$ is a simple homotopy equivalence with
$(f,g) \simeq (f'h,g'h\vert_N)$.

\noindent\textbf{(ii)}
Suppose for some $n \geqslant 2$ that $K_r(N)=0$ for all $r \neq n$. Then
\[
K_r(\ol{M}^-,N)~=~0~=~K_r(\ol{M}^+,N) \quad~\text{for all}~r \neq n+1~,
\]
and $K_n(N)$ is a stably finitely generated   free $\Z[H]$-module.
Moreover, we may define an object $x$ in $\NIL(\Z[H];\fB_1,\fB_2)$ by
\[
x ~:=~ (K_{n+1}(\ol{M}^-,N),K_{n+1}(\ol{M}^+,N),\rho_1,\rho_2)
\]
whose underlying modules satisfy
\begin{gather*}
[K_{n+1}(\ol{M}^-,N)]+[K_{n+1}(\ol{M}^+,N)]~=~[K_n(N)]~=~0 ~\in~ \wt{K}_0(\Z[H])~,\\[1ex]
[\Z[G_k]\otimes_{\Z[H]}K_{n+1}(\ol{M}^\epsilon,N)]~=~0 ~\in~ \wt{K}_0(\Z[G_k])\quad~(k=1,2)~.
\end{gather*}
If $(f',g'):(M',N') \to (X,Y)$ is obtained from $(f,g)$ by
a geometric cell-exchange killing an element $y_k \in K_{n+1}(\ol{M}^{\epsilon},N)$ $((k,\epsilon)=(1,-)$ or $(2,+))$ which generates a direct summand $\langle y_k \rangle \subseteq   K_{n+1}(\ol{M}^{\epsilon},N)$,
 then the corresponding object in $\NIL(\Z[H];\fB_1,\fB_2)$
\[
x'~:=~(K_{n+1}(\ol{M}'^-,N'),K_{n+1}(\ol{M}'^+,N'),\rho'_1,\rho'_2)
\]
is obtained from $x$ by an algebraic cell-exchange. Since $\pi_{n+1}(\ol{M}_k,N)=K_{n+1}(\ol{M}_k,N)$ by the relative Hurewicz theorem, there is a one-one correspondence between algebraic and geometric cell-exchanges killing elements $y_k$ generating direct summands $\langle y_k \rangle$.

\noindent\textbf{(iii)}
For any $n\geqslant 2$ it is possible to modify the
given $(f,g)$ by a finite sequence of geometric cell-exchanges  and
their formal inverses to obtain a pair (also denoted by $(f,g)$)
such that $K_r(N)=0$ for all $r \neq n$ as in {\rm (ii)}, and
hence a canonical equivalence class of  nilpotent objects
$x=(P_1,P_2,\rho_1,\rho_2)$ in $\NIL(\Z[H];\fB_1,\fB_2)$ such that
\[
[P_1]+[P_2]=0 \in \wt{K}_0(\Z[H])~,~
[\Z[G_k]\otimes_{\Z[H]}P_k]~=~0 \in \wt{K}_0(\Z[G_k])
\]
with $P_1:=K_{n+1}(\ol{M}^-,N)$, $P_2:=K_{n+1}(\ol{M}^+,N)$. Any $x'$ in the equivalence class of $x$ is realized by a map $(f',g'):(M',N')
\to (X,Y)$ with $f'$ simple-homotopic to $f$. The splitting obstruction of $f$ is the image of the Whitehead torsion
$\tau(f) \in \Wh(G)$, namely:
\begin{multline*}
\partial(\tau(f))~=~([P_1],[x])~=~([P_1],[P_1,P_2,\rho_1,\rho_2])\\[1ex]
~\in~ \mathrm{ker}(\wt{K}_0(\Z[H]) \to \wt{K}_0(\Z[G_1])\oplus \wt{K}_0(\Z[G_2])) \;\oplus\; \wt{\Nil}_0(\Z[H];\fB_1,\fB_2)~.
\end{multline*}
Thus $f$ is simply homotopic to a  split homotopy equivalence if and only if $\partial(\tau(f))=0$, if and only if $x$
is equivalent to $0$.

\noindent\textbf{(iv)}
The Whitehead group of $G=G_1*_HG_2$ fits into an exact sequence
\begin{multline*}\begin{CD}
\Wh(H) @>>> \Wh(G_1)\oplus \Wh(G_2) @>>> \Wh(G)\\[1ex]
@>{\partial}>> \wt{K}_0(\Z[H])\oplus \wt{\Nil}_0(\Z[H];\fB_1,\fB_2) @>>> \wt{K}_0(\Z[G_1])\oplus \wt{K}_0(\Z[G_2])~.
\end{CD}\end{multline*}
Furthermore, the homomorphism
\[
\partial ~:~ \Wh(G) ~\longrightarrow~ \wt{K}_0(\Z[H]) \oplus \wt{\Nil}_0(\Z[H];\fB_1,\fB_2)
~;~\tau(f) \longmapsto ([P_1],[P_1,P_2,\rho_1,\rho_2])
\]
satisfies that $\mathrm{proj}_2 \circ \partial: \Wh(G) \to \wt{\Nil}_0(\Z[H];\fB_1,\fB_2)$ is an epimorphism split by
\[
\iota~:~ \wt{\Nil}_0(\Z[H];\fB_1,\fB_2)  ~\longrightarrow~ \Wh(G)~;~
[P_1,P_2,\rho_1,\rho_2] ~\longmapsto~ \begin{bmatrix}
1 & \rho_2 \\ \rho_1 & 1\end{bmatrix}~.
\]
\end{thm}

\begin{dfn}\label{semisplit} Let $(X,Y)$ be a separating $\pi_1$-injective codimension 1 finite CW-pair.
A homotopy equivalence $f:M \to X$ from a finite CW-complex $M$  is \textbf{semi-split along $Y \subset X$} if
$f$ is simple homotopic to a map (also denoted by $f$) such that  for the corresponding map of pairs
$(f,g):(M,N) \to (X,Y)$ the relative homology kernel $\Z[H]$-modules
\[K_*(\ol{M}_2,N)~=~H_{*+1}((\wt{M}_2,\wt{N})
\to (\wt{X}_2,\wt{Y}))\]
vanish, which is equivalent to the induced $\Z[H]$-module morphisms
\[
\rho_2~:~K_*(\ol{M}^+,N) \longra K_*(\ol{M}^+,\ol{M}_2)
= \Z[G_2-H]\otimes_{\Z[H]} K_*(\ol{M}^-,N) ~,
\]
being isomorphisms. Equivalently, $f$ is semi-split along $Y$ if there is a semi-split object $x=(P_1,P_2,\rho_1,\rho_2)$
in the canonical  equivalence class of \ref{W}, that is, with $\rho_2:P_2 \to \fB_1P_1$ a $\Z[H]$-module isomorphism.
\end{dfn}

In particular, a split homotopy equivalence $f$ of separating pairs is semi-split.

\begin{thm}\label{Thm_TopSemisplit}
Let $(X,Y)$ be a separating $\pi_1$-injective codimension 1 finite CW-pair, with
$X=X_1\cup_YX_2$.  Suppose  that $H=\pi_1(Y)$ is a finite-index subgroup of $G_2=\pi_1(X_2)$. Every homotopy equivalence $f:\, M \to X$ with $M$ a finite CW-complex
is simple-homotopic to a homotopy equivalence  which is semi-split along $Y$.
\end{thm}

\begin{proof}
Let $x=(P_1,P_2,\rho_1,\rho_2)$ represent the canonical equivalence class of objects in $\NIL(\Z[H];\fB_1,\fB_2)$ associated to  $f$ in \fullref{W}(ii).
Since $H$ is of finite index in $G_2$, as in the proof of \fullref{Thm_HigherDKR}, we can define a semi-split object
\[
x'' ~:=~ (P_1,\fB_2 P_1,\rho_2\circ \rho_1,1)
\]
satisfying
\[
[x'']-[x]~=~[0,\fB_2 P_1,0,0]-[0,P_2,0,0] ~\in~ \Nil_0(\Z[H];\fB_1,\fB_2)~.
\]
By \fullref{W}(iii), the direct sum
\[
\fB_2 P_1 \oplus P_1~=~(\Z[G_2-H]\otimes_{\Z[H]}P_1) \oplus P_1~=~\Z[G_2] \otimes_{\Z[H]}P_1
\]
is a stably finitely generated free $\Z[G_2]$-module. Since $\Z[G_2]$ is
a finitely generated  free $\Z[H]$-module, $\fB_2 P_1 \oplus P_1$ is a stably finitely generated free $\Z[H]$-module.  So
\[
[\fB_2 P_1]-[P_2]~=~[\fB_2 P_1]+[P_1]~=~[\Z[G_2]\otimes_{\Z[H]}P_1]~=~0 \in \wt{K}_0(\Z[H])~.
\]
Therefore $x$ is equivalent to $x''$.
Thus, by \fullref{W}(iii), there is a homotopy equivalence $f''\,:M'' \to X$ simple-homotopic to $f$ realizing $x''$; note it is semi-split.
\end{proof}

\subsection*{Acknowledgements}\label{ackref}

We would like to thank the participants of the workshop \emph{Nil
Phenomena in Topology} (14--15 April 2007, Vanderbilt University),
where our interests intersected and motivated the development of
this paper.  We are grateful to the referee for making helpful comments and asking perceptive questions.  Moreover, Chuck Weibel helped us to formulate the
filtered colimit hypothesis in \fullref{maink}, and Dan Ramras
communicated the concept of almost-normal subgroup in \fullref{Defn_almostnormal} to the second-named author.

\bibliographystyle{gtart}
\bibliography{dkr_agt}
\end{document}